\documentclass[12pt]{amsart}

\usepackage{amsmath,amssymb,amsthm,url,a4wide, graphicx, float, mathrsfs, bbold}
\usepackage{xcolor}
\usepackage[all,color]{xy}

\DeclareMathOperator{\lat}{lat}
\newcommand{\FBL}{\mathrm{FBL}}
\newcommand{\one}{\mathbf{1}}

\DeclareMathOperator*{\bigfree}{\raisebox{-0.6ex}{\scalebox{2}{$\ast$}}}

\theoremstyle{plain}
\newtheorem{theorem}{Theorem}[section]
\newtheorem{lemma}[theorem]{Lemma}
\newtheorem{corollary}[theorem]{Corollary}
\newtheorem{proposition}[theorem]{Proposition}

\theoremstyle{definition}
\newtheorem{definition}[theorem]{Definition}
\newtheorem{remark}[theorem]{Remark}
\newtheorem{example}[theorem]{Example}

\title{Free Products of Banach Lattices}

\author[G. Mart\'inez-Fern\'andez]{Gonzalo Mart\'inez-Fern\'andez}
\address{Instituto de Ciencias Matem\'aticas\\
Consejo Superior de Investigaciones Cient\'ificas\\
C/ Nicol\'as Cabrera 13--15\\
28049 Madrid, Spain}
\email{gonzalo.martinez@icmat.es}

\author[P. Tradacete]{Pedro Tradacete}
\address{Instituto de Ciencias Matem\'aticas\\
Consejo Superior de Investigaciones Cient\'ificas\\
C/ Nicol\'as Cabrera 13--15\\
28049 Madrid, Spain}
\email{pedro.tradacete@icmat.es}

\date{}

\begin{document}

\begin{abstract}
We study free products, that is, coproducts, in the category of Banach lattices and contractive lattice homomorphisms. We give a concrete construction of the free product of an arbitrary family of Banach lattices as a quotient of a free Banach lattice, and prove its basic structural properties. We also establish stability results for sublattice embeddings and projective Banach lattices, and also analyze the behavior of quotient maps. For compact Hausdorff spaces $K_1$ and $K_2$ we identify $C(K_1)\ast C(K_2)$ lattice isomorphically with $C(K_1\ast K_2)$, where $K_1\ast K_2$ denotes the topological join, and we derive an explicit formula for the free product norm in this representation. We further discuss free factors of free Banach lattices, and exploit the existence of non-trivial homological spheres to show that a free Banach lattice can have free factors which are not isomorphic to free Banach lattices.
\end{abstract}

\subjclass[2020]{46B42, 46M15, 55P40} 

\keywords{Banach lattice; Free product; Topological join; Free Banach lattice}

\maketitle

\section{Introduction}

The purpose of this paper is to study free products in the category of Banach lattices and contractive lattice homomorphisms. Free products are among the most basic universal constructions in mathematics. They appear, in different contexts, as disjoint unions of sets, direct sums of vector spaces and modules, topological sums of spaces, free products of groups, free products of lattices, or $\ell_1$-sums of Banach spaces. In categorical language these are usually called \emph{coproducts}. Given a certain category $\mathcal{A}$ and $(A_i)_{i\in I}$ a family of objects of $\mathcal{A}$, a coproduct of this family is a new object $A$ together with morphisms $\iota_i:A_i \rightarrow A$, for $i\in I$, such that for every family of morphisms $\alpha_i:A_i\rightarrow B$, $i\in I$, there is a unique morphism $\alpha:A\rightarrow B$ such that $\alpha_i=\alpha\iota_i$ for every $i\in I$. Hence, this universal property can be considered as dual to that of \emph{products}. Recent work on coproducts in several different settings include for instance \cite{Carai2026Godel, Chatzinikolaou2023OperatorASystems, Chirvasitu2022Dearth, GluesingLuerssenJany2023QMatroids, HerfortHofmannRusso2024ProLie, IoanaSpaasVigdorovich2025TraceSpaces, Kavruk2014Nuclearity}.

Our starting point and motivation comes the theory of free Banach lattices. Free Banach lattices over sets were first introduced by de Pagter and Wickstead in \cite{projfree}, where the connection with projective Banach lattices was also established. Later, Avil\'es, Rodr\'iguez and the second author introduced the free Banach lattice generated by a Banach space in \cite{freebasico}. This construction extended the previous one and has been the object of intense research in the recent years (see for instance \cite{freebanlatlatproj, freebanlatlat, deHeviaTradacete2023ComplexFBL, GarciaSanchezLeungTaylorTradacete2025UpperP, JardonLaustsenTaylorTradaceteTroitsky2022Convexity, Oikhberg2024Geometry, freebanlat}). Free Banach lattices  provide a natural framework for constructing coproducts and other colimits in the category of Banach lattices.

Our first goal in this paper is to show that arbitrary free products of Banach lattices exist and admit a concrete representation as a certain quotient of a free Banach lattice. Once the construction is available, we establish the standard categorical properties expected from a coproduct: uniqueness up to lattice isometric isomorphism, commutativity, associativity, and minimality, in this case as norm-density of the sublattice generated by the canonical copies of the factors. We also describe the unit ball of the free product as a certain closed solid convex hull and obtain an explicit description for the free product of finitely many copies of $\mathbb{R}$.

We then focus on which properties of Banach lattices are stable under taking free products. Using the stability of isometric embeddings under push-outs in Banach lattices given in \cite{amalg}, we show that free products preserve sublattice embeddings. We also prove that projectivity is stable under free products and that the functor $E \mapsto \FBL[E]$ transforms free products of Banach spaces into free products of Banach lattices. In addition, we study how free products interact with quotients, and prove a $p$-convexity stability result by adapting the factorization method of \cite{factorpconvex}. We also show that the free Banach lattice generated by a lattice, introduced in \cite{freebanlatlat}, preserves lattice free products.

The main concrete representation result concerns Banach lattices of the form $C(K)$. If $K_1$ and $K_2$ are compact Hausdorff spaces, we show that $C(K_1)\ast C(K_2)$ is lattice isomorphic to $C(K_1\ast K_2)$, where $K_1\ast K_2$ denotes the topological join. This identifies a large and tractable class of free products within classical function spaces and yields an explicit formula for the norm in terms of the join parameter. For background on topological joins we refer to \cite{hatcher}.

As we mentioned, a purely categorical argument shows that coproducts of free objects produce free objects. A more subtle question is whether free factors of free objects are in turn free objects. Note that it is not in general true that a factor (direct summand) of a free $R$-module is free, although this holds if $R$ is sufficiently nice (a principal ideal domain). Of course, for vector spaces this holds trivially. Although not so easy to prove, the well known Nielsen-Schreier theorem \cite[Proposition 3.8, p.16]{combgrouptheo} states that every subgroup of a free group is free, so in particular, every factor of a free group is free. In the category of Banach spaces, free objects correspond to spaces of the form $\ell_1(A)$, and their only free factors (also called complemented subspaces) are isomorphic to $\ell_1(B)$, with $|B|\leq |A|$. This motivates the following question: is every free factor of $\FBL[\ell_1(A)]$ of the form $\FBL[\ell_1(B)]$ with $|B|\leq|A|$? More precisely, given a decomposition as a free product $\FBL[\ell_1]=X*Y$, is at least one of the factors isomorphic to $\FBL[\ell_1]$? Although the latter remains open, we give some partial results in that direction. In particular, we will show that free factors of a free Banach lattice need not be free Banach lattices. This will follow from a surprising connection between free products of $C(K)$'s and the Double Suspension Theorem due to James W. Cannon and Robert D. Edwards \cite{Doublesusptheo}.

The paper is organized as follows. Section \ref{sec:preliminaries} recalls the necessary background on free Banach lattices, push-outs, and topological joins. In Section \ref{sec:construction} we construct free products of families of Banach lattices and establish their basic categorical properties. Section \ref{sec:homomorphisms} studies the behavior of free products with respect to embeddings, surjections, and quotients, while Section~\ref{sec:stability} contains several stability results, including projectivity and compatibility with free Banach lattices. Section~\ref{sec:C(K)} is devoted to free products of $C(K)$-spaces, where we identify $C(K_1)\ast C(K_2)$ with a space of continuous functions on the join $K_1\ast K_2$ and discuss applications to AM-spaces, bases, and order-theoretic properties. In Section~\ref{sec:pconvex} we prove stability results related to $p$-convexity and upper $p$-estimates. In Section~\ref{sec:freefactors}, we consider free factors of free Banach lattices and use topological examples to produce non-free factors. Finally, Section~\ref{sec:freelattice} is devoted to the discussion of free products for free Banach lattices generated by lattices.

%

\section{Preliminaries}\label{sec:preliminaries}
\subsection{Free Banach lattices} Let $E$ be a Banach space. The free Banach lattice generated by $E$ is a Banach lattice $\FBL[E]$ together with an isometric linear map $\delta_E:E \to \FBL[E]$ such that every bounded linear operator $T:E \to Y$ into a Banach lattice $Y$ extends uniquely to a lattice homomorphism $\widehat{T}:\FBL[E] \to Y$ with $\widehat{T}\delta_E=T$ and $\|\widehat{T}\|=\|T\|$: 
\[
\xymatrix{\FBL[E] \ar[rd]^{\widehat T} \\ E \ar[u]^{\delta_E} \ar[r]^T & X}
\]

A concrete representation of $\FBL[E]$ is given in \cite{freebasico}: let $H[E]$ be the sublattice of positively homogeneous linear maps in $\mathbb{R}^{E^*}$ with the pointwise operations. The further sublattice of functions $f$ in $H[E]$ such that
\[
\|f\|=\sup\Big\{\sum_{i=1}^n|f(x_i^*)|:x_1^*,\ldots,x_n^*\in E^*,\text{ }\sup_{x\in B_E}\sum_{i=1}^n|x^*_i(x)|\leq 1,\text{ }n\in\mathbb{N}\Big\}<\infty,
\]
denoted by $H_1[E]$, is a Banach lattice when equipped with the lattice norm defined by the expression above. The closed sublattice of $H_1[E]$ generated by the evaluation functionals $\delta_x(x^*)=x^*(x)$, together with the linear isometry defined by $\delta_E(x):=\delta_x$ for every $x\in E$, can be shown to satisfy the universal property of the free Banach lattice $\FBL[E]$; see \cite[Theorem 2.4]{freebasico}.

We will need the following fact: $\FBL[E]$ has a strong order unit precisely when $\dim(E)<\infty$. In particular, $\FBL[\mathbb{R}^{n+1}]\simeq C(S^n)$ up to an equivalent lattice norm; see \cite[Proposition 9.1]{freebanlat}. Another remarkable property of $\FBL[\cdot]$ is its functoriality, that is, whenever we have a bounded linear map $T:E\to F$ of Banach spaces, it induces a unique lattice homomorphism \[
\overline T:\FBL[E]\to \FBL[F]
\]
making the following diagram commute
\[
\xymatrix{\FBL[E] \ar[r]^{\overline T} & \FBL[F] \\ E \ar[r]^T \ar[u]^{\delta_E} & F \ar[u]^{\delta_F}}
\]

The free Banach lattices above are examples of free objects in the category $\mathcal{BL}$ of Banach lattices with lattice homomorphisms. If we now restrict our attention to the subcategory of $p$-convex Banach lattices with $1<p\leq\infty$, we get the notion of the free $p$-convex Banach lattice generated by the Banach space $E$, first introduced in \cite{JardonLaustsenTaylorTradaceteTroitsky2022Convexity}: it is a $p$-convex Banach lattice with constant $1$, $\FBL^{(p)}[E]$, together with a linear isometric embedding $\delta_E:E\to\FBL^{(p)}[E]$ such that for every linear operator $T:E\rightarrow X$ to a $p$-convex Banach lattice $X$, we have a lattice homomorphism $\widehat T:\FBL^{(p)}[E]\to X$ with $\widehat T\delta_E=T$ and 
$\|\widehat T\|\leq M^{(p)}(X)\|T\|$,
where $M^{(p)}(X)$ is the $p$-convexity constant of $X$. Analogously, the notion of free Banach lattice with upper $p$-estimates $\FBL^{\uparrow p}[E]$ was also introduced in \cite{JardonLaustsenTaylorTradaceteTroitsky2022Convexity} and further developed in \cite{GarciaSanchezLeungTaylorTradacete2025UpperP}.

\subsection{Push-outs} Let $\mathcal{A}$ be a category, $A_0,A_1,A_2$ objects and morphisms $\alpha_i:A_0\rightarrow A_i$, $i=1,2$. A push-out diagram (or simply a push-out) is an object $PO=PO(\alpha_1,\alpha_2)$ together with morphisms $\beta_i:A_i\rightarrow PO$ such that $\beta_1\alpha_1=\beta_2\alpha_2$ and for every other pair of morphisms $\gamma_i:A_i\rightarrow B$ that make the diagram commutative, there is a unique morphism $\gamma:PO\rightarrow B$ making the following diagram commute.
$$\xymatrix{& & B \\
A_1 \ar[r]^{\beta_1} \ar@/^1.5pc/@[black][rru]^{\gamma_1} & PO \ar[ru]^\gamma \\ A_0 \ar[u]^{\alpha_1} \ar[r]^{\alpha_2} & A_2 \ar@/_1.5pc/@[black][ruu]^{\gamma_2} \ar[u]^{\beta_2}}$$

We work in the category $\mathcal{BL}$. In this context, we will call an \emph{isometric push-out} of $\alpha_1$ and $\alpha_2$ if we require the map $\gamma$ given above to satisfy $\|\gamma\|\leq \max\{\|\gamma_1\|,\|\gamma_2\|\}$. Existence of isometric push-outs in $\mathcal{BL}$ was established in \cite{amalg}. Moreover, \cite[Theorem~4.4]{amalg} stablishes the key property that if $\alpha_1$, $\alpha_2$ are isometric lattice embeddings, so are $\beta_1$ and $\beta_2$.


\subsection{Topological join} Let $\mathcal T_1$ and $\mathcal T_2$ be topological spaces. Their \emph{join} $\mathcal T_1\ast \mathcal T_2$ is defined as the quotient of $\mathcal T_1\times \mathcal T_2\times [0,1]$ obtained by identifying $(x,y_1,0)$ with $(x,y_2,0)$ for all $x\in \mathcal T_1$ and $y_1,y_2\in \mathcal T_2$, and identifying $(x_1,y,1)$ with $(x_2,y,1)$ for all $y\in \mathcal T_2$ and $x_1,x_2\in \mathcal T_1$. We denote the class of $(x,y,t)$ by $[x,y,t]$. When $K_1$ and $K_2$ are compact Hausdorff spaces, so is $K_1\ast K_2$. In this case, the join is associative up to homeomorphism (see \cite{hatcher}). 

The join of topological spaces is a generalization of well known constructions in topology: when $\mathcal T_2=\{p\}$ is a singleton, we recover the notion of cone of $\mathcal T_1$, $Cone(\mathcal T_1)=\mathcal T_1*\{p\}$. When $\mathcal T_2$ consists of two points (i.e., $\mathcal T_2=S^0$), the join is also called the suspension of $\mathcal T_1$, $S\mathcal T_1=\mathcal T_1*S^0$. One key result is that $S(S^n)=S^n*S^0=S^{n+1}$, and in general, $S^m*S^n=S^{m+n+1}$.

 We can naturally define the join of $n$ topological spaces as the quotient of $\mathcal T_1\times\ldots\times \mathcal T_n\times (B_1^n)_+$ with the proper identifications. The idea would be to regard each point as a formal convex combination of the form $t_1x_1+\ldots+t_nx_n$, with $\sum t_i=1$, $x_i\in \mathcal T_i$. More formally, we identify two pairs $(x_1,\ldots,x_i,\ldots,x_n,t)$ and $(x_1,\ldots,x_i',\ldots,x_n,t)$ whenever $t=(t_1,\ldots,t_n)$ and $t_i=0$ for that particular $i\in\{1,\ldots,n\}$. When $\mathcal T_1,\ldots,\mathcal T_n$ are compact Hausdorff spaces, their join $\mathcal T_1*\ldots*\mathcal T_n$ is always compact and Hausdorff, so $C(\mathcal T_1*\ldots*\mathcal T_n)$ will be on our list of $C(K)$ spaces. For the case when all spaces are compact Hausdorff, it is not hard to see that iteratively taking the join of $n$ spaces is the same, up to homeomorphism, as taking their join as a family of $n$ spaces. For instance, to show $(X*Y)*Z$ and $X*Y*Z$ are homeomorphic, it is enough to see that the map that sends $[[x,y,t],z,s]$ to $[x,y,z,((1-s)(1-t),(1-s)t,s)]$ is a continuous bijection. However, the join operation defined this way is not associative when we deal with arbitrary topological spaces. Anyway, we will focus on the join of only two compact Hausdorff spaces for the most part. For a reference on joins, the reader can see \cite{hatcher}.

\subsection{Notation} 
The symbol $\simeq$ is reserved for (lattice) isomorphisms between Banach spaces (lattices). For homeomorphisms between topological spaces, we use $\cong$. We shall use without further comment the standard fact that every lattice homomorphism $u:C(K)\to \mathbb{R}$ is of the form $u=c\,\delta_x$ for some $c\geq 0$ and some $x\in K$. 
For unexplained terminology on Banach lattices we refer the reader to the monographs \cite{burkinshaw,LT2,nieberg}.

\section{Construction and Basic Properties}\label{sec:construction}

\begin{definition}
Let $(X_i)_{i\in I}$ be a family of Banach lattices. A \emph{free product} of the family is a Banach lattice $X$ together with contractive lattice homomorphisms $\varphi_i:X_i\to X$ such that for every Banach lattice $Y$ and every family of lattice homomorphisms $T_i:X_i\to Y$ satisfying $\sup_{i\in I}\|T_i\|<\infty$, there exists a unique lattice homomorphism $T:X\to Y$ with
\[
T\varphi_i=T_i \qquad (i\in I)
\]
and
\[
\|T\|\leq \sup_{i\in I}\|T_i\|.
\]
\end{definition}

\begin{remark}
If $(X,(\varphi_i)_{i\in I})$ and $(Y,(\psi_i)_{i\in I})$ are free products of $(X_i)_{i\in I}$, then there exists a unique lattice isometric isomorphism $\Phi:X\to Y$ such that $\Phi\varphi_i=\psi_i$ for every $i\in I$.
\end{remark}

\begin{proof}
Applying the universal property twice, first with $T_i=\psi_i$ and then with $T_i=\varphi_i$, we obtain contractive lattice homomorphisms $\Phi:X\to Y$ and $\Psi:Y\to X$ with $\Phi\varphi_i=\psi_i$ and $\Psi\psi_i=\varphi_i$. Then $\Phi\Psi\psi_i=\psi_i$ for all $i\in I$, so uniqueness gives $\Phi\Psi=\mathrm{id}_Y$. Similarly, $\Psi\Phi=\mathrm{id}_X$.
\end{proof}

\begin{theorem}\label{thm:existence}
For every family $\bold{X}=(X_i)_{i\in I}$ of Banach lattices, the free product exists.
\end{theorem}

\begin{proof}
Let
\[
Z_\textbf{X}=\Big(\bigoplus_{i\in I} X_i\Big)_{\ell_1}
\] 
and let $\iota_i:X_i\to Z_\textbf{X}$ denote the canonical inclusions. Write $\delta_{Z_\textbf{X}}:Z_\textbf{X}\to \FBL[Z_\textbf{X}]$ for the canonical embedding. Let $J$ be the closed ideal of $\FBL[Z_\textbf{X}]$ generated by the elements
\[
\delta_{Z_\textbf{X}}(\iota_i|x|)-|\delta_{Z_\textbf{X}}(\iota_i x)|,
\qquad i\in I,\ x\in X_i.
\]
Set
\[
X=\FBL[Z_\textbf{X}]/J,
\]
let $q:\FBL[Z_\textbf{X}]\to X$ be the quotient map, and define
\[
\varphi_i=q\delta_{Z_\textbf{X}}\iota_i \qquad (i\in I).
\]
Clearly, $\varphi_i$ is contractive, and by definition of $J$, a lattice homomorphism.

Let now $Y$ be a Banach lattice and let $T_i:X_i\to Y$ be lattice homomorphisms with $M:=\sup_{i\in I}\|T_i\|<\infty$. Define
\[
S:Z_\textbf{X}\to Y,\qquad S((x_i)_{i\in I})=\sum_{i\in I}T_i x_i.
\]
Then $S$ is bounded and $\|S\|\leq M$. By the universal property of $\FBL[Z_\textbf{X}]$, there exists a lattice homomorphism $\widehat{S}:\FBL[Z_\textbf{X}]\to Y$ such that $\widehat{S}\delta_{Z_\textbf{X}}=S$ and $\|\widehat{S}\|=\|S\|$.

For every $i\in I$ and $x\in X_i$ we have
\[
\widehat{S}\big(\delta_{Z_\textbf{X}}(\iota_i|x|)-|\delta_{Z_\textbf{X}}(\iota_i x)|\big)
=T_i|x|-|T_i x|=0,
\]
so $J\subseteq \ker \widehat{S}$. Hence $\widehat{S}$ factors through a lattice homomorphism $T:X\to Y$ such that $Tq=\widehat{S}$. For each $i\in I$,
\[
T\varphi_i=Tq\delta_{Z_\textbf{X}}\iota_i=\widehat{S}\delta_{Z_\textbf{X}}\iota_i=S\iota_i=T_i.
\]

To prove uniqueness, suppose $R:X\to Y$ is another lattice homomorphism with $R\varphi_i=T_i$ for all $i\in I$. Then
\[
Rq\delta_{Z_\textbf{X}}\iota_i=T_i=Tq\delta_{Z_\textbf{X}}\iota_i
\qquad (i\in I),
\]
and therefore $Rq\delta_{Z_\textbf{X}}=Tq\delta_{Z_\textbf{X}}$ on $Z_\textbf{X}$. Since $q\delta_{Z_\textbf{X}}(Z_\textbf{X})$ generates a dense sublattice of $X$, we get $R=T$.

Finally,
\[
\|T\|\leq \|\widehat{S}\|=\|S\|\leq M.
\]
Conversely, since $T_i=T\varphi_i$ and each $\varphi_i$ is contractive, we have $\|T_i\|\leq \|T\|$ for every $i\in I$. Thus
\[
\|T\|=\sup_{i\in I}\|T_i\|.
\]
\end{proof}

In view of the previous results, we will use the symbol $\bigfree_{i\in I} X_i$ to denote \textit{the} free product of the family $(X_i)_{i\in I}$. Each $X_i$ is a \textit{free factor} of the free product. For the homomorphism $T$ extending all $T_i$ we will usually write $\bigfree_{i\in I} T_i$, too.

\begin{remark}\label{rem:isometric-copies}
For every $i\in I$, the canonical map $\varphi_i:X_i\to \bigfree_{j\in I}X_j$ is a lattice isometric embedding.
\end{remark}

\begin{proof}
Fix $i\in I$ and define $T_i=\mathrm{id}_{X_i}$ and $T_j=0$ for $j\neq i$. Let $T:\bigfree_{j\in I}X_j\to X_i$ be the corresponding free product homomorphism. Then $T\varphi_i=\mathrm{id}_{X_i}$, so for every $x\in X_i$,
\[
\|x\|=\|T\varphi_i(x)\|\leq \|\varphi_i(x)\|\leq \|x\|.
\]
\end{proof}

\begin{remark}\label{rem:bigfree-formula}
Every element of $\lat(\bigcup_{i\in I}\varphi_i(X_i))\subset \bigfree_{i\in I}X_i$ can be written as a lattice-linear expression
\[
\Phi(\varphi_{i_1}x_1,\ldots,\varphi_{i_n}x_n)
\]
for suitable $x_j\in X_{i_j}$. If $(T_i)_{i\in I}$ is a uniformly bounded family of lattice homomorphisms, then
\[
\bigfree_{i\in I}T_i\bigl(\Phi(\varphi_{i_1}x_1,\ldots,\varphi_{i_n}x_n)\bigr)
=
\Phi(T_{i_1}x_1,\ldots,T_{i_n}x_n).
\]
\end{remark}

The following result is rather straightforward but serves as warm up: it states free products enjoy commutativity and associativity. The proof is a general categorical argument relying solely on uniqueness.

\begin{proposition}
Let $(X_i)_{i\in I}$ be a family of Banach lattices.

\begin{enumerate}
\item If $\sigma:I\to I$ is a bijection and $Y_i=X_{\sigma(i)}$, then
\[
\bigfree_{i\in I} X_i \simeq \bigfree_{i\in I} Y_i
\]
lattice isometrically.
\item If $\mathcal{P}$ is a partition of $I$, then
\[
\bigfree_{J\in \mathcal{P}}\Big(\bigfree_{i\in J} X_i\Big)\simeq \bigfree_{i\in I} X_i
\]
lattice isometrically.
\end{enumerate}
\end{proposition}

\begin{proof}
Both statements follow directly from uniqueness of the free product. In (1) one simply relabels the canonical inclusions. In (2) the compositions of the canonical inclusions
\[
X_i\longrightarrow \bigfree_{i\in J}X_i \longrightarrow \bigfree_{J\in\mathcal{P}}\Big(\bigfree_{i\in J}X_i\Big)
\]
satisfy the universal property of the free product of the family $(X_i)_{i\in I}$.
\end{proof}

\begin{proposition}\label{prop:dense-sublattice}
The sublattice generated by $\bigcup_{i\in I}\varphi_i(X_i)$ is dense in $\bigfree_{i\in I}X_i$.
\end{proposition}

\begin{proof}
Let
\[
Z=\overline{\lat\Big(\bigcup_{i\in I}\varphi_i(X_i)\Big)} \subseteq \bigfree_{i\in I}X_i.
\]
Given any uniformly bounded family of lattice homomorphisms $T_i:X_i\to Y$, the free product homomorphism $\bigfree_{i\in I}T_i:\bigfree_{i\in I}X_i\to Y$ restricts to a lattice homomorphism $(\bigfree_{i\in I}T_i)|_Z:Z\to Y$ extending all $T_i$. The uniqueness of $(\bigfree_{i\in I}T_i)|_Z$ follows because its values are determined on the dense sublattice generated by the canonical copies of the $X_i$. Hence $Z$, together with the same canonical inclusions, is itself a free product of $(X_i)_{i\in I}$. By uniqueness there exists a lattice isometric isomorphism $\Phi:Z\to \bigfree_{i\in I}X_i$ fixing each canonical copy of $X_i$. Therefore $\Phi$ is the identity on a dense sublattice of $Z$, hence on all of $Z$. Since $\Phi$ is surjective, we conclude that $Z=\bigfree_{i\in I}X_i$.
\end{proof}

\begin{remark} An alternative proof of the above statement goes as follows: if there were an element $z\in \bigfree_{i\in I}X_i\backslash Z$, then take a positive functional $h\in (\bigfree_{i\in I} X_i)^*_+$ such that $h|_Z=0$ and $h(z)\neq 0$. A standard application of Kakutani's representation theorem for $AL$-spaces yields that the functional $h$ induces a lattice homomorphism $\bigfree_{i\in I} X_i\rightarrow L_1(\mu)$ that vanishes on $Z$ but not at $z$. This is in contradiction with the uniqueness of the extension.
\end{remark}

\begin{remark}
In view of last proposition, the density character of a free product of Banach lattices $\bigfree_{i\in I} X_i$ with $\text{dens}(X_i)=\kappa_i$ is 
\begin{align*}
\text{dens}(\bigfree_{i\in I}X_i)=\sum_{i\in I}\kappa_i=\max\{|I|,\sup_{i\in I}\kappa_i\}.
\end{align*}

The inequality $\geq$ is evident. On the other hand, if we take dense subsets $D_i$ of $\varphi_i(X_i)$ of cardinality $\kappa_i$, $lat\{D_i:i\in I\}$ (which has cardinal $\sum_{i\in I} \kappa_i$),  is dense in $\bigfree_{i\in I} X_i$ by last proposition.
\end{remark}

\begin{proposition}\label{prop:complemented}
For every $i\in I$, the canonical copy $\varphi_i(X_i)\subseteq \bigfree_{j\in I}X_j$ is the range of a contractive lattice projection. Moreover, the closed linear span of $\bigcup_{i\in I}\varphi_i(X_i)$ is linearly isometric to $\big(\bigoplus_{i\in I}X_i\big)_{\ell_1}$ and is complemented in $\bigfree_{i\in I}X_i$.
\end{proposition}

\begin{proof}
Fix $i\in I$ and define $T_i=\mathrm{id}_{X_i}$ and $T_j=0$ for $j\neq i$. If $T:\bigfree_{j\in I}X_j\to X_i$ is the corresponding free product homomorphism, then
\[
P_i=\varphi_i T
\]
is a contractive lattice projection onto $\varphi_i(X_i)$.

For the second statement, let
\[
Z_\textbf{X}=\Big(\bigoplus_{i\in I}X_i\Big)_{\ell_1}
\]
with its coordinatewise order, and let $\iota_i:X_i\to Z_\textbf{X}$ denote the canonical inclusions. By the universal property of the free product there exists a contractive lattice homomorphism
\[
T:\bigfree_{i\in I}X_i\to Z_\textbf{X}
\]
such that $T\varphi_i=\iota_i$ for all $i\in I$. On the other hand, the universal property of the $\ell_1$-sum in the category of Banach spaces yields a contractive linear operator
\[
S:Z_\textbf{X}\to \bigfree_{i\in I}X_i
\]
with $S\iota_i=\varphi_i$. Then
\[
TS\iota_i=T\varphi_i=\iota_i
\qquad (i\in I),
\]
so $TS=\mathrm{id}_{Z_\textbf{X}}$. Therefore $ST$ is a contractive projection onto $S(Z_\textbf{X})$, and $S(Z_\textbf{X})$ is linearly isometric to $Z_\textbf{X}$.
\end{proof}

In a similar way that for any family of Banach spaces $(E_i)$ the $\ell_1$-sum norm is the largest norm one can define on the algebraic sum $\bigoplus E_i$ preserving the norm of each $E_i$, the free product norm, denoted by $\|\cdot\|_*$, is maximal among lattice norms.

\begin{proposition}\label{prop:maximal-norm}
Let
\[
Z=\lat\{\varphi_i(X_i):i\in I\}\subseteq \bigfree_{i\in I}X_i.
\]
Suppose that $\|\cdot\|$ is a lattice norm on $Z$ such that
\[
\|\varphi_i(x)\|\leq \|x\|_{X_i}
\qquad (x\in X_i,\ i\in I).
\]
Then
\[
\|z\|\leq \|z\|_{\ast}
\qquad (z\in Z),
\]
where $\|\cdot\|_{\ast}$ denotes the free product norm.
\end{proposition}

\begin{proof}
Let $\widehat{Z}$ be the completion of $(Z,\|\cdot\|)$ and let $j:Z\to \widehat{Z}$ be the inclusion. For each $i\in I$, the map
\[
T_i:X_i\to \widehat{Z},\qquad T_i(x)=j(\varphi_i(x)),
\]
is a contractive lattice homomorphism. Hence there exists a contractive lattice homomorphism
\[
T:\bigfree_{i\in I}X_i\to \widehat{Z}
\]
with $T\varphi_i=T_i$ for all $i\in I$. On the dense sublattice $Z$, the map $T$ is simply the formal identity. Therefore
\[
\|z\|=\|Tz\|\leq \|z\|_{\ast}
\qquad (z\in Z).
\]
\end{proof}

\begin{remark}\label{rem:freeveclat}
For every Banach space $E$, the free Banach lattice $\FBL[E]$ can be regarded as the completion of the free vector lattice over $E$ under the largest lattice norm making the canonical map from $E$ contractive. It is not clear whether an analogous description holds for free products of Banach lattices. More precisely, if one forms the vector-lattice free product of a family $(X_i)_{i\in I}$, it is not known whether its Banach-lattice completion coincides with the dense sublattice $Z=\lat\{\varphi_i(X_i):i\in I\}$ inside $\bigfree_{i\in I}X_i$.
\end{remark}

So far, we have been able to prove the aforementioned basic properties of free products relying only on their universal property. However, it is combining the universal property with a nice representation of the space the key to go beyond. For certain Banach lattices $X_i$, we can do much better than just a quotient of $\FBL[(\bigoplus_{i\in I} X_i)_{\ell_1}]$. Next result, which will be generalized in the following sections, is an appropriate starting point in this direction. 

\begin{proposition}\label{prop:Rn}
For every $n\in \mathbb{N}$, the free product $\bigfree_{i=1}^n\mathbb{R}$ is lattice isomorphic to $C((S_{\ell_\infty^n})_+)$, where $(S_{\ell_\infty^n})_+$ denotes the positive part of the unit sphere of $\ell_\infty^n$. Under this identification, the free product norm is given by
\[
\|f\|_*=\inf\Big\{a_1+\cdots+a_n:\ |f|\leq \sum_{i=1}^n a_i \delta_{e_i},\ a_i\geq 0\Big\}.
\] 

\end{proposition}

\begin{proof}
By construction, $\bigfree_{i=1}^n\mathbb{R}=\FBL[\ell_1^n]/J$, where $J$ is the closed ideal generated by the elements $\delta_{|ce_i|}-|\delta_{ce_i}|$, with $c\in \mathbb R$, $i=1,\ldots,n$. But we can lattice isomorphically identify $\FBL[\ell_1^n]$ with $C(S_{\ell_\infty^n})$ (see Proposition 5.3, \cite{projfree}) via the restriction to the sphere $S_{\ell_{\infty}^n}$. So we can regard $J$ as the set of continuous functions of $S_{\ell_\infty^n}$ that vanish in the common zeroes of the functions $\delta_{|ce_i|}-|\delta_{ce_i}|$ restricted the sphere. Note that given $(y_1,\ldots,y_n)\in S_{\ell_\infty^n}$, 
\[
|\delta_{|ce_i|}-|\delta_{ce_i}||(y_1,\ldots,y_n)=||c|y_i-|cy_i||=|c||y_i-|y_i||,
\]
which vanishes exactly when $y_i$ is positive. So the common zeroes are just the positive vectors of the sphere, $(S_{\ell_\infty^n})_+$, and the functions in $J$ are exactly the ones that vanish on the given set. Thus, the restriction to the positive part of the sphere induces a lattice isomorphism $\FBL[\ell_1]/J\longrightarrow C((S_{\ell_\infty^n})_+)$.

Now define, for $f\in C((S_{\ell_\infty^n})_+)$,
\[
p(f)=\inf\Big\{a_1+\cdots+a_n:\ |f|\leq \sum_{i=1}^n a_i \delta_{e_i},\ a_i\geq 0\Big\}.
\]
This is a lattice norm, and it is equivalent to the supremum norm. Moreover, for each $i$ and each $c\in \mathbb{R}$,
\[
p(c\delta_{e_i})=|c|,
\]
because $|c\delta_{e_i}|\leq a_1\delta_{e_1}+\cdots+a_n\delta_{e_n}$ implies $|c|\leq a_i$ by evaluation at $e_i\in (S_{\ell_\infty^n})_+$, while the reverse inequality is immediate.

Thus $p$ preserves the norm of each canonical copy of $\mathbb{R}$. By Proposition~\ref{prop:maximal-norm}, we have $p(f)\leq \|f\|_{\ast}$. Conversely, if $|f|\leq \sum_{i=1}^n a_i \delta_{e_i}$, then
\[
\|f\|_{\ast}\leq \sum_{i=1}^n a_i\|\delta_{e_i}\|_{\ast}=\sum_{i=1}^n a_i.
\]
Taking the infimum yields $\|f\|_{\ast}\leq p(f)$, and hence $p=\|\cdot\|_{\ast}$.
\end{proof}

As a consequence, the free product of Banach lattices is not order continuous. In fact, it need not be order complete. In particular, the free product of order complete Banach lattices need not be order complete as the following shows: $\mathbb{R}$ is order continuous (and order complete). However, $\mathbb{R}*\mathbb{R}\simeq C((S_{\ell_\infty^2})_+)\simeq C([0,1])$ is not order complete, because $[0,1]$ is not extremally disconnected; see \cite[Proposition 1.a.4]{LT2}

The way we expressed the free product norm in last proposition suggests the following

\begin{proposition}\label{prop:unit-ball}
The unit ball of $\bigfree_{i\in I}X_i$ is the closed solid convex hull of $\bigcup_{i\in I}\varphi_i(B_{X_i})$.
\end{proposition}

\begin{proof}
Let
\[
A=\bigcup_{i\in I}\varphi_i(B_{X_i})
\]
and let $C$ be the closed solid convex hull of $A$. Since
\[
A\subseteq B_{\bigfree_{i\in I}X_i},
\]
we have
\[
C\subseteq B_{\bigfree_{i\in I}X_i}.
\]

Let
\[
Z=\lat\{\varphi_i(X_i):i\in I\}.
\]
Every element of $Z$ belongs to $\lambda C$ for some $\lambda>0$: indeed, any lattice-linear expression in elements of $A$ is dominated by a finite positive combination of elements of $A$, by repeated use of the inequalities
\[
|x\pm y|\leq |x|+|y|,\qquad |x\vee y|\leq |x|+|y|,\qquad |x\wedge y|\leq |x|+|y|.
\]
Hence the Minkowski functional
\[
p(z)=\inf\{\lambda>0:\ z\in \lambda C\},
\qquad z\in Z,
\]
is a lattice norm on $Z$.

If $x\in X_i$, then $\varphi_i(x)\in \|x\|\,C$, so $p(\varphi_i(x))\leq \|x\|$. Thus $p$ is dominated by the original norm on each factor. Proposition~\ref{prop:maximal-norm} therefore gives
\[
p(z)\leq \|z\|_{\ast}
\qquad (z\in Z).
\]
Since
\[
C\subseteq B_{\bigfree_{i\in I}X_i},
\]
the reverse inequality $\|z\|_{\ast}\leq p(z)$ is immediate. Hence $p=\|\cdot\|_{\ast}$ on $Z$, so $B_{\bigfree_{i\in I}X_i}\cap Z=C\cap Z$. Because $Z$ is dense and $C$ is closed, it follows that the unit ball of $\bigfree_{i\in I}X_i$ is exactly $C$.
\end{proof}

\section{Stability of homomorphisms}\label{sec:homomorphisms}

Given a uniformly bounded family of lattice homomorphisms $T_i:X_i\to Y_i$, we denote by
\[
\overline{\bigfree}_{i\in I}T_i:\bigfree_{i\in I}X_i\to \bigfree_{i\in I}Y_i
\]
the unique lattice homomorphism extending the maps $\psi_iT_i$, where $\psi_i:Y_i\to \bigfree_{i\in I}Y_i$ are the canonical inclusions.

\begin{proposition}\label{prop:embedding-stability}
Let $Y_i\subseteq X_i$ be closed sublattices and let $j_i:Y_i\hookrightarrow X_i$ be the inclusion maps. Then
\[
J:=\overline{\bigfree}_{i\in I}j_i:\bigfree_{i\in I}Y_i\to \bigfree_{i\in I}X_i
\]
is a lattice isometric embedding.
\end{proposition} 

\begin{proof}
We first write the proof for the case $I=\{1,2\}$, then deal with the general case.

We have a commutative diagram
$$
\xymatrix{X_1 \ar@/^1.5pc/@[black][rr]^{\varphi_1} & & X_1*X_2  \\ Y_1 \ar[r]^{\psi_1}  \ar[u]^{j_1} & Y_1*Y_2 \ar[ru]^{J} \\ 0 \ar[u] \ar[r] & Y_2 \ar[u]^{\psi_2} \ar[r]^{j_2} & X_2 \ar@/_1.5pc/@[black][uu]^{\varphi_2}}
$$
where $\psi_i:Y_i\rightarrow Y_1*Y_2$ are the canonical inclusions, and $J=j_1\overline\ast j_2$. Now we can consider the push-outs of the maps $j_i$ and $\psi_i$, which we denote by $PO_i$, for $i=1,2$:
$$
\xymatrix{X_1 \ar[r]^{S_1}  \ar@/^1.5pc/@[black][rr]^{\varphi_1} & PO_1 & X_1*X_2  \\ Y_1 \ar[r]^{\psi_1}  \ar[u]^{j_1} & Y_1*Y_2 \ar[u]^{R_1} \ar[r]^{R_2} \ar[ru]^{J} & PO_2 \\ 0 \ar[u] \ar[r] & Y_2 \ar[u]^{\psi_2} \ar[r]^{j_2} & X_2 \ar[u]^{S_2} \ar@/_1.5pc/@[black][uu]_{\varphi_2}}
$$

Finally, we take the push-out of $R_1$ and $R_2$, which we call $PO$, and let $W=T_1R_1*T_2R_2$:
$$
\xymatrix{&&& PO \\ X_1 \ar[r]^{R_1}  \ar@/^1.5pc/@[black][rr]^{\varphi_1} & PO_1 \ar@/^1.5pc/@[black][rru]^{T_1} & X_1*X_2 \ar[ru]^{W} \\ Y_1 \ar[r]^{\psi_1}  \ar[u]^{j_1} & Y_1*Y_2 \ar[u]^{S_1} \ar[r]^{S_2} \ar[ru]^{J} & PO_2 \ar@/_1.5pc/@[black][ruu]^{T_2} \\ 0 \ar[u] \ar[r] & Y_2 \ar[u]^{\psi_2} \ar[r]^{j_2} & X_2 \ar[u]^{R_2} \ar@/_1.5pc/@[black][uu]_{\varphi_2}}
$$ 

Recall that push-outs of Banach lattices preserve isometries \cite[Theorem 4.4]{amalg}. Then, since $j_2$ and $j_1$ are isometries, so are $S_1$ and $S_2$, and thus, $T_1$ and $T_2$. 

In order to see that $J$ is an isometry, it is enough to prove $T_1S_1=WJ$, because $T_1S_1$ is an isometry and all arrows are contractive. Let us prove, then, that those compositions agree: indeed, $T_1S_1\psi_1=T_1R_1j_1=W\varphi_1j_1=WJ\psi_1$, where the last two equalities follow from the definition of $J$ and $W$ respectively. Similarly, $T_1S_1\psi_2=WJ\psi_2$. But $lat\{\psi_1(Y_1),\psi_2(Y_2)\}$ is norm dense in $Y_1*Y_2$, so $T_1S_1=WJ$.

An elementary induction argument establishes the result for free products of finitely many Banach lattices. For the general case, it is sufficient to show that for every finite lattice linear expression $z=\Phi(\varphi_{i_1}(y_1),\ldots,\varphi_{i_n}(y_n))\in\bigfree_{i\in I} Y_i$, $\|Jz\|=\|z\|$. But note that $\overline{lat}(Y_{i_1},\ldots,Y_{i_m})\simeq Y_{i_1}*\ldots * Y_{i_m}$ isometrically, so we have the following commutative square
$$
\xymatrix{ \overline{lat}(Y_{i_1},\ldots,Y_{i_m}) \ar[r]^{J_|} & \overline{lat}(X_{i_1},\ldots,X_{i_m}) \\ Y_{i_1}*\ldots *Y_{i_m} \ar[r]^{j_{i_1}\overline\ast\ldots\overline\ast j_{i_n}} \ar[u] & X_{i_1}*\ldots *X_{i_m} \ar[u]} 
$$
The lower horizontal arrow is an isometry by what we have already seen. Hence, $\|J(z)\|=\|z\|$.
\end{proof}

For a family $\bold{X}=(X_i)_{i\in I}$ of Banach lattices, let
\[
Z_{\bold{X}}=\Big(\bigoplus_{i\in I}X_i\Big)_{\ell_1},
\]
let $\iota_i:X_i\to Z_{\bold{X}}$ be the canonical inclusions, let $\delta_{Z_\bold{X}}:Z_{\bold{X}}\to \FBL[Z_{\bold{X}}]$ be the canonical map, and let $J_{\bold{X}}$ denote the closed ideal of $\FBL[Z_{\bold{X}}]$ generated by
\[
\delta_{Z_\bold{X}}(\iota_i|x|)-|\delta_{Z_{\bold{X}}}(\iota_i x)|,
\qquad i\in I,\ x\in X_i.
\]
Then
\[
\bigfree_{i\in I}X_i \simeq \FBL[Z_{\bold{X}}]/J_{\bold{X}}.
\]

For a Banach lattice $X$, let $\beta_X:\FBL[X]\to X$ be the lattice homomorphism induced by the identity map on $X$. The kernel $\ker\beta_X$ agrees with the closed ideal $J_{\bold X}$ generated by the elements $|\delta_X x|-\delta_X|x|:x\in X$. Obviously, $X$ is itself a free product of the one element family $(X)$, with embedding $id_X$. By uniqueness, we have a lattice isomorphism \[
\Phi:FBL[X]/J_{\bold X}\rightarrow X
\]
such that $\Phi\varphi_1=id_X$. But $\Phi q$ agrees with $\beta$ on $\delta(X)$, so in fact $\Phi q=\beta$. And since $\Phi$ is injective, $\ker\beta\subset \ker q=J_{\bold X}$. To see the other inclusion, $\beta$ vanishes at every element of the form $|\delta x|-\delta|x|$. Then, $J_{\bold X}\subset \ker\beta$ by continuity. Since $\beta_X$ is a lattice projection onto the canonical copy of $X$, we also have the decomposition
\[
\FBL[X]=\delta_X(X)\oplus \ker \beta_X.
\]

\begin{lemma}
Let $T:X\to Y$ be a lattice homomorphism between Banach lattices, and let $\overline{T}:\FBL[X]\to \FBL[Y]$ be the induced lattice homomorphism. Then
\[
\overline{T}(\delta_X(X))\subseteq \delta_Y(Y)
\qquad\text{and}\qquad
\overline{T}(\ker \beta_X)\subseteq \ker \beta_Y.
\]
\end{lemma}

\begin{proof}
The first inclusion is immediate from the identity
\[
\overline{T}\,\delta_X x=\delta_Y(Tx),
\qquad x\in X.
\]

For the second inclusion, note that $\ker\beta_X$ is the closed ideal generated by the elements $|\delta_X x|-\delta_X|x|$. Since
\[
\overline{T}\bigl(|\delta_X x|-\delta_X|x|\bigr)
=
|\delta_Y(Tx)|-\delta_Y|Tx|
\in \ker\beta_Y,
\]
the ideal generated by these elements is mapped into $\ker\beta_Y$, and therefore so is its closure.
\end{proof}

\begin{remark}
Let $\bold{X}=(X_i)_{i\in I}$, $\bold{Y}=(Y_i)_{i\in I}$ and $T_i:X_i\to Y_i$ lattice homomorphisms with $\sup_i\|T_i\|<\infty$ and
\[
T=\bigoplus_{i\in I}T_i:Z_{\bold{X}}\to Z_{\bold{Y}},
\]
then the induced homomorphism $\overline{T}:\FBL[Z_{\bold{X}}]\to \FBL[Z_{\bold{Y}}]$ satisfies $\overline{T}(J_\bold{X})\subseteq J_{\bold{X}}$. Consequently, it induces a lattice homomorphism
\[
S:\bigfree_{i\in I}X_i\to \bigfree_{i\in I}Y_i,
\]
and by uniqueness this induced map is exactly $\overline{\bigfree}_{i\in I}T_i$.
\end{remark}

\begin{proposition}\label{prop:surjective-factors}
Let $T_i:X_i\to Y_i$ be lattice homomorphisms with $\sup_{i\in I}\|T_i\|<\infty$. If
\[
\overline{\bigfree}_{i\in I}T_i:\bigfree_{i\in I}X_i\to \bigfree_{i\in I}Y_i
\]
is surjective, then each $T_i$ is surjective. The converse holds provided the direct-sum operator $\bigoplus_{i\in I}T_i:\Big(\bigoplus_{i\in I}X_i\Big)_{\ell_1}\to \Big(\bigoplus_{i\in I}Y_i\Big)_{\ell_1}$
is surjective.
\end{proposition}

\begin{proof}
Let
\[
S=\overline{\bigfree}_{i\in I}T_i:\bigfree_{i\in I}X_i\to \bigfree_{i\in I}Y_i.
\]
For each $i\in I$, let
\[
U_i:\bigfree_{j\in I}X_j\to X_i
\qquad\text{and}\qquad
V_i:\bigfree_{j\in I}Y_j\to Y_i
\]
be the free product homomorphisms obtained by taking the identity on the $i$th factor and the zero map on all other factors. Then
\[
P_i^X=\varphi_iU_i
\qquad\text{and}\qquad
P_i^Y=\psi_iV_i
\]
are the canonical contractive lattice projections onto the copies of $X_i$ and $Y_i$, respectively.

On the dense sublattice generated by the canonical copies, one has
\[
P_i^Y S=\psi_i T_i U_i.
\]
Indeed, both sides agree on each factor, so they agree everywhere by density. Now fix $y\in Y_i$. Since $S$ is surjective, there exists $x\in \bigfree_{j\in I}X_j$ with $Sx=\psi_i y$. Hence
\[
\psi_i y=P_i^Y(Sx)=\psi_i T_i(U_i x).
\]
Because $\psi_i$ is injective, we obtain
\[
y=T_i(U_i x),
\]
which proves that $T_i$ is surjective.

For the converse, surjectivity of the direct sum implies surjectivity of the induced homomorphism between the corresponding free Banach lattices (\cite[Proposition 3.2]{freebanlat}), and passing to the quotients by $J_{\bold{X}}$ and $J_{\bold{Y}}$ yields surjectivity of $\overline{\bigfree}_{i\in I}T_i$.
\end{proof}

\begin{remark}
The maps $T_n:\mathbb{R}\to \mathbb{R}$, $T_nx=x/n$, are surjective for every $n\in \mathbb{N}$, but $\oplus_{n=1}^\infty T_n:\ell_1\to \ell_1$ is not surjective. Consequently,
\[
\overline{\bigfree}_{n=1}^\infty T_n:\bigfree_{n=1}^\infty\mathbb{R}\to \bigfree_{n=1}^\infty\mathbb{R}
\]
is not surjective either.
\end{remark}

\begin{lemma}\label{lem:FBL-quotient}
Let $E$ be a Banach space and let $F\subseteq E$ be a closed subspace. Then
\[
\FBL[E/F]\simeq \FBL[E]/\widehat{F}
\]
lattice isometrically, where $\widehat{F}$ is the closed ideal of $\FBL[E]$ generated by $\delta_E(F)$.
\end{lemma}

\begin{proof}
Let $q:E\to E/F$ be the quotient map and let $Q:\FBL[E]\to \FBL[E]/\widehat{F}$ be the canonical quotient homomorphism. Since $\delta_E(F)\subseteq \widehat{F}$, the formula
\[
\widetilde{\delta}(q x)=Q(\delta_E x)
\qquad (x\in E)
\]
defines a contractive linear map $\widetilde{\delta}:E/F\to \FBL[E]/\widehat{F}$.

Let $Y$ be a Banach lattice and let $T:E/F\to Y$ be a bounded linear operator. Then $S=Tq:E\to Y$ extends uniquely to a lattice homomorphism $\widehat{S}:\FBL[E]\to Y$ with $\widehat{S}\delta_E=S$ and $\|\widehat{S}\|=\|S\|=\|T\|$. Since $\widehat{S}(\delta_E x)=0$ for every $x\in F$, the ideal $\widehat{F}$ is contained in $\ker \widehat{S}$, so $\widehat{S}$ factors through a lattice homomorphism
\[
R:\FBL[E]/\widehat{F}\to Y
\]
with $RQ=\widehat{S}$. Then
\[
R\widetilde{\delta}(q x)=RQ(\delta_E x)=\widehat{S}\delta_E x=Tq x,
\]
so $R$ extends $T$ and $\|R\|=\|T\|$. Uniqueness follows because $\widetilde{\delta}(E/F)$ generates a dense sublattice of $\FBL[E]/\widehat{F}$.
\end{proof}

\begin{corollary}\label{cor:freeprod-quotients}
Let $(X_i)_{i\in I}$ be a family of Banach lattices and let $J_i\subseteq X_i$ be closed ideals. Then
\[
\bigfree_{i\in I}(X_i/J_i)\simeq
\Big(\bigfree_{i\in I}X_i\Big)\Big/ J
\]
lattice isometrically, where $J$ is the closed ideal of $\bigfree_{i\in I}X_i$ generated by the images of the ideals $J_i$.
\end{corollary}

\begin{proof}
Let
\[
Z_{\bold{X}}=\Big(\bigoplus_{i\in I}X_i\Big)_{\ell_1},
\qquad
J_0=\Big(\bigoplus_{i\in I}J_i\Big)_{\ell_1}\subseteq Z_{\bold{X}}.
\]
Then $Z_{\bold{X}}/J_0$ is canonically isometric to $\big(\bigoplus_{i\in I}(X_i/J_i)\big)_{\ell_1}$. Let $\widehat{J_0}$ be the closed ideal of $\FBL[Z_{\bold{X}}]$ generated by $\delta_X(J_0)$. By Lemma~\ref{lem:FBL-quotient},
\[
\FBL[Z_{\bold{X}}/J_0]\simeq \FBL[Z_{\bold{X}}]/\widehat{J_0}.
\]

Let us call $Y_i=X_i/J_i$ and $\bold{Y}=(Y_i)_{i\in I}$. With this notation in mind, \[
\FBL[Z_{\bold Y}]\simeq \FBL[Z_{\bold X}]/\widehat {J_0},\qquad  J_{\bold Y}\simeq(J_{\bold{X}}+\widehat{J_0})/\widehat{J_0}.
\] Therefore
\begin{align*}
\bigfree_{i\in I}(X_i/J_i)
&\simeq \FBL[Z_{\bold Y}]/J_{\bold Y} \\
&\simeq (\FBL[Z_{\bold{X}}]/\widehat{J_0})/\big((J_\bold{X}+\widehat{J_0})/\widehat{J_0}\big) \\
&\simeq (\FBL[Z_{\bold{X}}]/J_{\bold{X}})\big/\big((J_{\bold{X}}+\widehat{J_0})/J_{\bold{X}}\big).
\end{align*}
Since $\FBL[Z_{\bold{X}}]/J_{\bold{X}}\simeq \bigfree_{i\in I}X_i$, the conclusion follows with
\[
J=(J_{\bold{X}}+\widehat{J_0})/J_{\bold{X}},
\]
which is exactly the closed ideal generated by the canonical images of the $J_i$.
\end{proof}
\begin{proposition}
Let $T_i:X_i\to Y_i$ be lattice homomorphisms. Then
\[
\overline{\bigfree}_{i\in I}T_i:\bigfree_{i\in I}X_i\to \bigfree_{i\in I}Y_i
\]
is a quotient map if and only if each $T_i$ is a quotient map.
\end{proposition}

\begin{proof}
Assume first that each $T_i$ is a quotient map. Let $J_i=\ker T_i$ and let $q_i:X_i\to X_i/J_i$ be the quotient map. Then $T_i$ factors as an isometric lattice isomorphism from $X_i/J_i$ onto $Y_i$ composed with $q_i$. By Corollary~\ref{cor:freeprod-quotients}, the induced map
\[
\overline{\bigfree}_{i\in I}q_i:\bigfree_{i\in I}X_i\to \bigfree_{i\in I}(X_i/J_i)
\]
is the canonical quotient map by the closed ideal generated by the copies of the $J_i$. Composing with the induced lattice isomorphism
\[
\bigfree_{i\in I}(X_i/J_i)\simeq \bigfree_{i\in I}Y_i
\]
shows that $\overline{\bigfree}_{i\in I}T_i$ is a quotient map.

Conversely, suppose that $\overline{\bigfree}_{i\in I}T_i$ is a quotient map. Fix $i\in I$, and let
\[
P_i^X:\bigfree_{j\in I}X_j\to \varphi_i(X_i),
\qquad
P_i^Y:\bigfree_{j\in I}Y_j\to \psi_i(Y_i)
\]
be the canonical contractive lattice projections. Then
\[
P_i^Y\overline{\bigfree}_{j\in I}T_j=\psi_i T_i U_i
\]
on the dense sublattice generated by the factors, hence everywhere. Since $\overline{\bigfree}_{j\in I}T_j(B_{\bigfree X_j})$ is dense in $B_{\bigfree Y_j}$, applying $P_i^Y$ shows that $\psi_i(T_i(B_{X_i}))$ is dense in $\psi_i(B_{Y_i})$. Because $\psi_i$ is an isometric embedding, $\overline{T_i(B_{X_i})}=B_{Y_i}$, so $T_i$ is a quotient map.
\end{proof}

\section{Stability under free products}\label{sec:stability}

In this section, we explore which properties of Banach lattices are inherited by their free product.

\begin{definition}
A Banach lattice $X$ is said to be \emph{projective} if for every Banach lattice $Z$, every closed ideal $J\subseteq Z$, every lattice homomorphism $T:X\to Z/J$, and every $\varepsilon>0$, there exists a lattice homomorphism $\widetilde{T}:X\to Z$ such that $q\widetilde{T}=T$ and
\[
\|\widetilde{T}\|\leq (1+\varepsilon)\|T\|,
\]
where $q:Z\to Z/J$ denotes the quotient map (see \cite{projfree}).
\end{definition}
The notion of projective Banach lattice was first introduced in \cite{freebanlatlatproj} 

\begin{proposition}\label{prop:projective}
The free product $\bigfree_{i\in I}X_i$ is projective if and only if every $X_i$ is projective.
\end{proposition}

\begin{proof}
Assume first that every $X_i$ is projective. Let $q:Z\to Z/J$ be a quotient map of Banach lattices, let $\varepsilon>0$, and let
\[
T:\bigfree_{i\in I}X_i\to Z/J
\]
be a lattice homomorphism. For each $i\in I$, the map $T_i=T\varphi_i:X_i\to Z/J$ lifts to a lattice homomorphism $\widetilde{T}_i:X_i\to Z$ such that
\[
\|\widetilde{T}_i\|\leq (1+\varepsilon)\|T_i\|\leq (1+\varepsilon)\|T\|.
\]
By the universal property of the free product, there exists a lattice homomorphism
\[
\widetilde{T}:\bigfree_{i\in I}X_i\to Z
\]
with $\widetilde{T}\varphi_i=\widetilde{T}_i$ for all $i\in I$. Then
\[
q\widetilde{T}\varphi_i=q\widetilde{T}_i=T_i=T\varphi_i
\]
for all $i\in I$, so $q\widetilde{T}=T$ by uniqueness. Moreover,
\[
\|\widetilde{T}\|=\sup_{i\in I}\|\widetilde{T}_i\|
\leq (1+\varepsilon)\|T\|.
\]

Conversely, assume that $\bigfree_{i\in I}X_i$ is projective. Fix $i_0\in I$, let $q:Z\to Z/J$ be a quotient map, and let $S:X_{i_0}\to Z/J$ be a lattice homomorphism. Define $T_i=0$ for $i\neq i_0$ and $T_{i_0}=S$. Let
\[
T:\bigfree_{i\in I}X_i\to Z/J
\]
be the corresponding free product homomorphism. By projectivity, for every $\varepsilon>0$ there exists a lift $\widetilde{T}:\bigfree_{i\in I}X_i\to Z$ with $\|\widetilde{T}\|\leq (1+\varepsilon)\|T\|$. Then $\widetilde{T}\varphi_{i_0}$ is a lift of $S$, and
\[
\|\widetilde{T}\varphi_{i_0}\|\leq \|\widetilde{T}\|\leq (1+\varepsilon)\|T\|=(1+\varepsilon)\|S\|.
\]
Thus $X_{i_0}$ is projective.
\end{proof}

We need the following well-known fact, which we only include for convenience of the reader.

\begin{lemma}\label{lem:comp-l1}
Let $E$ and $F$ be Banach spaces. If $\ell_1$ is complemented in $E\oplus F$, then $\ell_1$ is complemented in either $E$ or $F$.
\end{lemma}

\begin{proof}
By \cite[Theorem~10]{diestel}, the space $E\oplus F$ contains a complemented copy of $\ell_1$ if and only if $(E\oplus F)^*$ contains a copy of $c_0$. By \cite[Theorem~2.4.11]{albiackalton}, this is equivalent to the existence of a weakly unconditionally Cauchy series in $(E\oplus F)^*$ that is not unconditionally convergent. Write such a series as
\[
\sum_{n=1}^\infty (x_n^*+y_n^*),
\qquad x_n^*\in E^*,\ y_n^*\in F^*.
\]
Projecting onto the two coordinates and using \cite[Proposition~2.4.7]{albiackalton}, both $\sum x_n^*$ and $\sum y_n^*$ are weakly unconditionally Cauchy. At least one of them fails to be unconditionally convergent, since otherwise their sum would be unconditionally convergent as well. Applying \cite[Theorem~2.4.11]{albiackalton} again, either $E^*$ or $F^*$ contains a copy of $c_0$, and therefore either $E$ or $F$ contains a complemented copy of $\ell_1$.
\end{proof}

\begin{proposition}
If $X$ and $Y$ are Banach lattices with order continuous duals, then the dual of $X\ast Y$ is order continuous.
\end{proposition} 

\begin{proof}
By \cite[Theorem~4.69]{burkinshaw}, a Banach lattice has order continuous dual if and only if it does not contain a lattice copy of $\ell_1$. Hence $X$ and $Y$ do not contain complemented linear copies of $\ell_1$. By Lemma~\ref{lem:comp-l1}, the Banach space $X\oplus Y$ does not contain a complemented copy of $\ell_1$ either. Using \cite[Theorem~9.20]{freebanlat}, it follows that $\FBL[X\oplus Y]$ contains no lattice copy of $\ell_1$.

Suppose now that $(X\ast Y)^*$ were not order continuous. Then $X\ast Y$ would contain a lattice copy of $\ell_1$. Since $\ell_1$ is projective as a Banach lattice by \cite[Theorem~11.11]{projfree}, such a copy would lift through the quotient map from $\FBL[X\oplus Y]\simeq \FBL[X]\ast \FBL[Y]$ onto $X\ast Y$, yielding a lattice copy of $\ell_1$ inside $\FBL[X\oplus Y]$, a contradiction.
\end{proof}

The previous proposition fails for infinite free products. Indeed, Proposition~\ref{prop:complemented} shows that $\bigfree_{n=1}^\infty\mathbb{R}$ contains a complemented copy of $\ell_1$, and therefore \cite[Theorem~4.69]{burkinshaw} implies that it also contains a lattice copy of $\ell_1$.

If $(E_i)_{i\in I}$ is a family of Banach spaces, then the free product of the free Banach lattices $\FBL[E_i]$ is again a free Banach lattice, namely, the one generated by the ``free product'' of $(E_i)_{i\in I}$ in the Banach space category.

\begin{proposition}\label{prop:fbl-preserves}
Let $(E_i)_{i\in I}$ be a family of Banach spaces. Then
\[
\FBL\Big[\Big(\bigoplus_{i\in I}E_i\Big)_{\ell_1}\Big]\simeq \bigfree_{i\in I}\FBL[E_i]
\]
lattice isometrically.
\end{proposition}

\begin{proof}
Let
\[
Z_\textbf{E}=\Big(\bigoplus_{i\in I}E_i\Big)_{\ell_1},
\] 
let $\iota_i:E_i\to Z_\textbf{E}$ be the canonical inclusions, and let $\delta_i:E_i\to \FBL[E_i]$ and $\delta_{Z_\textbf{E}}:Z_\textbf{E}\to \FBL[Z_\textbf{E}]$ be the canonical embeddings. For each $i\in I$, the universal property of $\FBL[E_i]$ yields a contractive lattice homomorphism
\[
\overline{\iota}_i:\FBL[E_i]\to \FBL[Z_\textbf{E}]
\]
such that 
\[
\xymatrix{\FBL[E_i] \ar[r]^{\overline{\iota_i}} & \FBL[Z_{\bold{E}}] \\ E_i \ar[u]^{\delta_i} \ar[r]^{\iota_i} & Z_{\bold{E}} \ar[u]^{\delta_{Z_{\bold{E}}}}}
\]
commutes.

We claim that $\FBL[Z_\textbf{E}]$, together with the maps $\overline{\iota}_i$, is a free product of the family $(\FBL[E_i])_{i\in I}$. Indeed, let $Y$ be a Banach lattice and let $T_i:\FBL[E_i]\to Y$ be lattice homomorphisms with $\sup_i\|T_i\|<\infty$. Define $S_i=T_i\delta_i:E_i\to Y$ and
\[
S=\bigoplus_{i\in I}S_i:Z_\textbf{E}\to Y.
\]
By the universal property of $\FBL[Z_\textbf{E}]$, there exists a lattice homomorphism $\widehat{S}:\FBL[Z_\textbf{E}]\to Y$ such that $\widehat{S}\delta_{Z_\textbf{E}}=S$. Then
\[
\widehat{S}\,\overline{\iota}_i\delta_i
 =\widehat{S}\delta_{Z_\textbf{E}}\iota_i
 =S\iota_i
 =S_i
 =T_i\delta_i.
\]
By uniqueness in $\FBL[E_i]$, we have $\widehat{S}\,\overline{\iota}_i=T_i$ for all $i\in I$. Uniqueness of $\widehat{S}$ follows because $\delta_{Z_\textbf{E}}(Z_\textbf{E})$ generates a dense sublattice of $\FBL[{Z_\textbf{E}}]$.
\end{proof}

\begin{proposition}
Let $X$ be a Banach lattice and let $E$ be a Banach space. If $J_X$ denotes the closed ideal of $\FBL[X\oplus_1 E]$ generated by the elements
\[
|\delta_{X\oplus_1 E}(x,0)|-\delta_{X\oplus_1 E}(|x|,0),
\qquad x\in X,
\]
then
\[
X\ast \FBL[E]\simeq \FBL[X\oplus_1 E]/J_X
\]
lattice isometrically.
\end{proposition} 

\begin{proof}
Let $\delta_E:E\to \FBL[E]$ and $\delta_{X\oplus_1 E}:X\oplus_1 E\to \FBL[X\oplus_1 E]$ denote the canonical embeddings, and let $\iota_2:E\to X\oplus_1 E$ be the inclusion into the second coordinate. By the universal property of $\FBL[E]$, the composition $\delta_{X\oplus_1 E}\iota_2$ extends to a contractive lattice homomorphism
\[
\overline{\iota}_2:\FBL[E]\to \FBL[X\oplus_1 E].
\]
Let $q:\FBL[X\oplus_1 E]\to \FBL[X\oplus_1 E]/J_X$ be the quotient map. We claim that the quotient, together with the maps
\[
x\mapsto q\delta_{X\oplus_1 E}(x,0)
\qquad\text{and}\qquad
u\mapsto q\overline{\iota}_2(u),
\]
is a free product of $X$ and $\FBL[E]$.

Let $Z$ be a Banach lattice, let $T_1:X\to Z$ be a lattice homomorphism, and let $T_2:\FBL[E]\to Z$ be a lattice homomorphism. Put $S_2=T_2\delta_E:E\to Z$. By the universal property of $\FBL[X\oplus_1 E]$, the operator
\[
T_1\oplus S_2:X\oplus_1 E\to Z
\]
extends uniquely to a lattice homomorphism
\[
\widetilde{T}:\FBL[X\oplus_1 E]\to Z.
\]
For $x\in X$ we have
\[
\widetilde{T}\bigl(\delta_{X\oplus_1 E}(|x|,0)-|\delta_{X\oplus_1 E}(x,0)|\bigr)
=
T_1|x|-|T_1x|=0,
\]
so $J_X\subseteq \ker \widetilde{T}$. Thus $\widetilde{T}$ factors through a lattice homomorphism
\[
T:\FBL[X\oplus_1 E]/J_X\to Z.
\]
By construction,
\[
Tq\delta_{X\oplus_1 E}(x,0)=T_1x
\qquad\text{and}\qquad
Tq\overline{\iota}_2=T_2.
\]
Uniqueness follows because $q\delta_{X\oplus_1 E}(X\oplus_1 E)$ generates a dense sublattice of the quotient. Hence the quotient satisfies the universal property of $X\ast \FBL[E]$.
\end{proof}

\section{Free Products of $C(K)$-Spaces}\label{sec:C(K)}

Throughout this section, $K_1$ and $K_2$ denote compact Hausdorff spaces, and
\[
\varphi_1:C(K_1)\to C(K_1)\ast C(K_2),
\qquad
\varphi_2:C(K_2)\to C(K_1)\ast C(K_2)
\]
denote the canonical inclusions. Recall that the join $ K_1* K_2$ is the quotient space of $ K_1\times K_2\times [0,1]$ when we identify points of the form $(x,y,0)$ and $(x,y',0)$ for each $x\in K_1$ and $y,y'\in K_2$ and similarly points of the form $(x,y,1)$ and $(x',y,1)$, for $x,x'\in K_1$ and $y\in K_2$, where $ K_1 \times K_2\times [0,1]$ is endowed with the product topology.

\begin{proposition}\label{prop:strong-unit}
The element
\[
e=\varphi_1(\one_{K_1})+\varphi_2(\one_{K_2})\in C(K_1)\ast C(K_2)
\]
is a strong order unit.
\end{proposition}

\begin{proof}
Let
\[
Z=\lat\big(\varphi_1(C(K_1))\cup \varphi_2(C(K_2))\big).
\]
The element $e$ is clearly a strong order unit in $Z$, so the gauge
\[
\|z\|_e=\inf\{\lambda>0:\ |z|\leq \lambda e\},
\qquad z\in Z,
\]
is a lattice norm on $Z$. If $f\in C(K_1)$, then
\[
|\varphi_1(f)|\leq \|f\|_\infty\, \varphi_1(\one_{K_1})\leq \|f\|_\infty\, e,
\]
and therefore $\|\varphi_1(f)\|_e\leq \|f\|_\infty$. The same estimate holds for $\varphi_2(g)$ with $g\in C(K_2)$. By Proposition~\ref{prop:maximal-norm},
\[
\|z\|_e\leq \|z\|_{\ast}
\qquad (z\in Z).
\]
Conversely, if $|z|\leq \lambda e$, then
\[
\|z\|_{\ast}\leq \lambda \|e\|_{\ast}
\leq 2\lambda,
\]
because $\|\varphi_i(\one_{K_i})\|_{\ast}=1$ by Remark~\ref{rem:isometric-copies}. Hence
\[
\|z\|_{\ast}\leq 2\|z\|_e
\qquad (z\in Z).
\]
Thus the norms $\|\cdot\|_e$ and $\|\cdot\|_{\ast}$ are equivalent on the dense sublattice $Z$, so they have the same completion. In particular, every element of $C(K_1)\ast C(K_2)$ is dominated by a multiple of $e$, and therefore $e$ is a strong order unit.
\end{proof}

For $a,b\in \mathbb{R}$, let
\[
L(a,b)(t)=(1-t)a+tb,
\qquad t\in [0,1].
\]
If $f\in C(K_1)$ and $g\in C(K_2)$, define $L(f,g)\in C(K_1\ast K_2)$ by
\[
L(f,g)[x,y,t]=(1-t)f(x)+tg(y).
\]

\begin{theorem}\label{thm:CK-join}
Let $K_1$ and $K_2$ be compact Hausdorff spaces. The spaces $C(K_1)\ast C(K_2)$ and $C(K_1\ast K_2)$ are lattice isomorphic (with constant 2). Under this identification, the canonical copies of $C(K_1)$ and $C(K_2)$ are given by
\[
\varphi_1(g)=L(g,0),
\qquad
\varphi_2(h)=L(0,h).
\]
and the free product norm of a function $f\in C(K_1*K_2)$ is given by
\[
\|f\|_*
=\inf\Big\{a+b:\ a,b\in\mathbb R_+,\ |f[x,y,t]|\leq (1-t)a+tb\ \text{for all } [x,y,t]\in K_1\ast K_2\Big\}.
\]

\end{theorem}

\begin{proof}
By Proposition~\ref{prop:strong-unit}, the free product $C(K_1)\ast C(K_2)$ has a strong order unit. By Kakutani's representation theorem, there exists a compact Hausdorff space $F$ such that
\[
C(K_1)\ast C(K_2)\simeq C(F),
\]
where \[
F=\{u:C(K_1)*C(K_2)\to\mathbb R:u\text{ is a lattice homomorphism and } u(e)=1\}
\]
with the product topology and $e=\varphi_1(\bold 1_{K_1})+\varphi_2(\bold 1_{K_2})$ is the strong unit from Proposition \ref{prop:strong-unit}. We claim that $F\cong K_1*K_2$.

Let
\[
\pi:K_1\times K_2\times [0,1]\to K_1\ast K_2
\]
be the quotient map. 

Let $u\in F$, and set $u_i=u\circ \varphi_i$ for $i=1,2$. Each $u_i$ is a lattice homomorphism from $C(K_i)$ to $\mathbb{R}$, so $u_i=c_i\delta_{x_i}$ for some $c_i\geq 0$ and $x_i\in K_i$. If $c_i=0$, we can choose $x_i\in K_i$ arbitrarily. Since
\[
1=u(e)=u_1(\one_{K_1})+u_2(\one_{K_2})=c_1+c_2,
\]
we may define
\[
\Phi(u)=[x_1,x_2,c_2]\in K_1\ast K_2.
\]
The identifications of points in the join ensure that $\Phi$ is well defined when $c_2=0$ or $c_2=1$.

The map $\Phi$ is bijective. Indeed, given $[x_1,x_2,t]\in K_1\ast K_2$, the lattice homomorphisms $(1-t)\delta_{x_1}$ on $C(K_1)$ and $t\delta_{x_2}$ on $C(K_2)$ extend uniquely to a lattice homomorphism
\[
u:C(K_1)\ast C(K_2)\to \mathbb{R}
\]
with $u(e)=1$, and $\Phi(u)=[x_1,x_2,t]$. Conversely, if $\Phi(u)=\Phi(v)=[x_1,x_2,t]$, then
\[
u_1=(1-t)\delta_{x_1}=v_1,
\qquad
u_2=t\delta_{x_2}=v_2,
\]
and uniqueness in the free product gives $u=v$.

It remains to prove continuity. Let $(u^\alpha)$ be a net in $F$ converging to $u\in F$. Write
\[
u_i^\alpha=c_i^\alpha\delta_{x_i^\alpha},
\qquad
u_i=c_i\delta_{x_i}
\]
whenever the corresponding functional is non-zero, and choose $x_i^\alpha\in K_i$ arbitrarily if $u_i^\alpha=0$. Then
\[
c_i^\alpha=u_i^\alpha(\one_{K_i})\longrightarrow u_i(\one_{K_i})=c_i
\qquad (i=1,2).
\]
If $0<c_i<1$, then eventually $c_i^\alpha>0$, and for every $f\in C(K_i)$ we have
\[
c_i^\alpha f(x_i^\alpha)=u_i^\alpha(f)\longrightarrow u_i(f)=c_i f(x_i).
\]
Since $c_i^\alpha\to c_i>0$, it follows that $f(x_i^\alpha)\to f(x_i)$ for every $f\in C(K_i)$, hence $x_i^\alpha\to x_i$. Therefore
\[
\Phi(u^\alpha)\to \Phi(u)
\]
whenever $0<c_2<1$.

If $c_2=0$, then $c_1=1$ and $x_1^\alpha\to x_1$ by the same argument applied to $u_1^\alpha$. A basic neighborhood of $\Phi(u)$ in $K_1\ast K_2$ has the form
\[
\pi(U\times K_2\times [0,\varepsilon))
\]
for some neighborhood $U$ of $x_1$ in $K_1$. Since $x_1^\alpha\to x_1$ and $c_2^\alpha\to 0$, we eventually have
\[
\Phi(u^\alpha)\in \pi(U\times K_2\times [0,\varepsilon)).
\]
The case $c_2=1$ is analogous. Thus $\Phi$ is continuous.

Since $F$ is compact and $K_1\ast K_2$ is Hausdorff, $\Phi$ is a homeomorphism. Therefore
\[
C(K_1)\ast C(K_2)\simeq C(F)\simeq C(K_1\ast K_2).
\]

To see what the canonical inclusions look like under this identification, fix $f\in C(K_1)$ and $[x,y,t]\in K_1\ast K_2$. Let
\[
u_{x,y,t}:C(K_1)\ast C(K_2)\to \mathbb{R}
\]
be the lattice homomorphism corresponding, via the argument above, to the point $[x,y,t]$. Then
\[
\varphi_1(f)[x,y,t]
=u_{x,y,t}(\varphi_1(f))
=((1-t)\delta_x\ast t\delta_y)(\varphi_1(f))
=(1-t)f(x)
=L(f,0)[x,y,t].
\]
The formula for $\varphi_2(g)$ is identical.

Lastly, we compute the free product norm. Define for every $f\in C(K_1*K_2)$ the following lattice seminorm
\[
p(f)=\inf\Big\{a+b:\ a,b\in\mathbb R_+,\ |f|\leq L(a,b)\Big\}.
\]
This is a lattice norm, and it is equivalent to the supremum norm because
\[
\|f\|_\infty\leq p(f)\leq 2\|f\|_\infty.
\]

If $|f|\leq L(a,b)$, then
\[
L(a,b)=L(a,0)+L(0,b)=a\,\varphi_1(\one_{K_1})+b\,\varphi_2(\one_{K_2}),
\]
and therefore
\[
\|f\|_{\ast}\leq a\|\varphi_1(\one_{K_1})\|_{\ast}+b\|\varphi_2(\one_{K_2})\|_{\ast}=a+b.
\]
Taking the infimum yields $\|f\|_{\ast}\leq p(f)$.

On the other hand, if $g\in C(K_1)$ then $|L(g,0)|\leq L(a,b)$ implies $|g(x)|\leq a$ for every $x\in K_1$, hence
\[
p(L(g,0))\leq\|g\|_\infty.
\]
Similarly,
\[
p(L(0,h))\leq\|h\|_\infty
\qquad (h\in C(K_2)).
\]
Thus $p$ agrees with the original norms on the two canonical copies. By Proposition~\ref{prop:maximal-norm},
\[
p(f)\leq \|f\|_{\ast}
\]
on the dense sublattice generated by those copies, and therefore on all of $C(K_1\ast K_2)$ by continuity. Hence $p=\|\cdot\|_{\ast}$.
\end{proof}

\begin{remark} 
Given a Banach lattice $X$, let $\mathbf{n}(X)$ denote the least cardinality of a set that generates $X$ as a Banach lattice. By \cite[Proposition~9.6]{freebanlat}, one has $\mathbf{n}(\FBL[E])=\dim(E)$ when $E$ is finite-dimensional and $\mathbf{n}(\FBL[E])=\infty$ when $E$ is infinite-dimensional. In general,
\[
\max\{\mathbf{n}(X),\mathbf{n}(Y)\}\leq \mathbf{n}(X\ast Y)\leq \mathbf{n}(X)+\mathbf{n}(Y),
\]
and the second inequality can be strict since $\mathbb{R}^2\ast \mathbb{R}\simeq C([0,1])\simeq \mathbb{R}\ast \mathbb{R}$. The fact that $\mathbb{R}^2\ast \mathbb{R}\simeq C([0,1])$ is a direct consequence of Theorem \ref{thm:CK-join}.
\end{remark}

The class of $AM$-spaces is well-behaved under free products too. Let $K$ be a compact Hausdorff space. If $\mathcal{M}$ is a set of triples $(x,y,\alpha)$ with $x\neq y\in K$ and $\alpha\geq 0$, let
\[
C(K;\mathcal{M})=\{f\in C(K): f(x)=\alpha f(y)\text{ for all }(x,y,\alpha)\in \mathcal{M}\}.
\]
Every $AM$-space with strong unit is lattice isometric to a space of this form. Moreover, if $(x,y,\alpha_1)$ and $(x,y,\alpha_2)$ both belong to $\mathcal{M}$ with $\alpha_1>\alpha_2$, then necessarily $f(x)=f(y)=0$ for every $f\in C(K;\mathcal{M})$, so those relations are equivalent to the pair $(x,y,0)$ and $(y,x,0)$. Thus one may assume without loss of generality that $\mathcal{M}$ does not contain any pair of triples of the form $(x,y,\alpha_1)$ and $(x,y,\alpha_2)$ with $\alpha_1>\alpha_2$.

Kakutani's theorem actually yields the following (\cite[Theorem 1.b.5]{LT2}):
Let $Y$ be a sublattice of $C(K)$, and let $\mathcal{M}$ be the set of triples $(x,y,\alpha)$ such that $f(x)=\alpha f(y)$ for every $f\in Y$. Then
\[
\overline{Y}=C(K;\mathcal{M}).
\]

\begin{proposition}\label{prop:AM-freeproduct}
Let $\mathcal{M}_1$ and $\mathcal{M}_2$ be relation sets on compact spaces $K_1$ and $K_2$, respectively. Then
\[
C(K_1;\mathcal{M}_1)\ast C(K_2;\mathcal{M}_2)\simeq C(K_1\ast K_2;\mathcal{M}),
\]
where $\mathcal M$ is the set of all triples
\[
([x,z,t],[y,w,s],\lambda)
\]
such that $(x,y,\alpha)\in \mathcal{M}_1$ whenever $t<1$, $(z,w,\beta)\in \mathcal{M}_2$ whenever $t>0$, and
\[
\lambda=\alpha(1-t)+\beta t.
\]
If $\lambda=0$, $s$ can be arbitrary; otherwise
\[
s=\frac{\beta t}{\alpha(1-t)+\beta t}
\]
for $0<t<1$. If $t=0$ and $(w,z,0)\in\mathcal{M}_2$, $s$ can take any value; else $s=0$. If $t=1$ and $(y,x,0)\in\mathcal{M}_1$, $s$ can be arbitrary; otherwise $s=1$.

\end{proposition}

\begin{proof}
Let $A$ be the set of all triples from the statement of the proposition. Since 
\[
C(K_1;\mathcal{M}_1)\ast C(K_2;\mathcal{M}_2)
\]
is a sublattice of $C(K_1\ast K_2)$ by Proposition \ref{prop:embedding-stability}, it must be of the form
\[
C(K_1*K_2;\mathcal M),
\]
where $\mathcal M$ is a certain set of triples. Our goal is to show that $\mathcal M=A$. Take $([x,z,t],[y,w,s],\lambda)\in\mathcal{M}$ with $\lambda>0$ first.

Let us suppose that $0<t<1:$
we must have \[
f(x)=\frac{\lambda(1-s)}{1-t}f(y), \qquad g(z)=\frac{\lambda s}{t}g(w)
\] for all $f\in C(K;\mathcal{M}_1)$ and $g\in C(L;\mathcal{M}_2)$. So $(x,y,\alpha=\frac{\lambda(1-s)}{1-t})\in\mathcal{M}_1$ and $(z,w,\beta=\frac{\lambda s}{t})\in\mathcal{M}_2$, where $\alpha,\beta$ are minimal. 

Suppose now that $t=0$: it must be
\[
f(x)=\lambda(1-s)f(y), \qquad 0=\lambda sg(w).
\] When $s=0$ too, this is equivalent to $(x,y,\alpha=\lambda)\in\mathcal{M}_1$. If $s>0$, it is equivalent to $(x,y,\lambda (1-s))\in\mathcal {M}_1$ and $(w,x,0)\in \mathcal{M}_2$ ($g(w)=0$). 

The case $t=1$ is analogous.

Now consider the case $\lambda=0$: if $0<t<1$, \[
f(x)=0=g(z)
\]
must hold for every $f\in C(K_1;\mathcal{M}_1)$ and $g\in C(K_2;\mathcal{M}_2)$, this means, $(x,y,0)\in\mathcal{M}_1$ and $(z,w,0)\in\mathcal{M}$.

The case $t=0$ is equivalent to $(x,y,0)\in\mathcal{M}_1$. Analogous for $t=1$.

Summarizing, we have proved $\mathcal{M}\subset A$ so far. The other inclusion is direct, recalling the positive homogeneity of lattice linear combinations, and the fact that $lat(Y)=Y^{\vee}-Y^{\vee}$ when $Y$ is a subspace. We conclude by the observation preceding this proposition.
\end{proof}

The join representation also gives a useful topological interpretation of free factors of $C(K)$-spaces.

\begin{proposition}\label{prop:cone-factor}
A Banach lattice of the form $C(K)$ has a free factor isomorphic to $\mathbb{R}$ if and only if $K$ is homeomorphic to a cone. Moreover,
\[
C(K)\simeq C(K)\ast \mathbb{R}
\]
if and only if $K\cong \operatorname{Cone}(K)$.
\end{proposition}

\begin{proof}
If $C(K)\simeq Y\ast \mathbb{R}$, then $Y$ is an $AM$-space with strong unit, so $Y\simeq C(L)$ for some compact $L$. By Theorem~\ref{thm:CK-join},
\[
C(K)\simeq C(L)\ast \mathbb{R}\simeq C(\operatorname{Cone}(L)),
\]
which is equivalent to $K\cong \operatorname{Cone}(L)$. The converse is immediate from the same theorem.
\end{proof}

\begin{example}
Examples of compact spaces that are not cones are abundant; for instance, $S^0$ is not a cone because cones are connected and contractible. On the other hand, $[0,1]=\operatorname{Cone}(\{0\})$ is a cone but it is not homeomorphic to its own cone. A standard example of a compactum $K$ with $K\cong \operatorname{Cone}(K)$ is the Hilbert cube, which we represent as
\[
Q=\{(0,(x_n))\in\ell_2:|x_n|\leq 1/n\}.
\]
for convenience of notation. We then have a continuous map $Cone(Q)\rightarrow co(e_1,Q)$ given by $[(x_n),t]\mapsto (t,(1-t)x_1,(1-t)x_2,\ldots)$. It is bijective, hence a homeomorphism, because $\operatorname{Cone}(Q)$ is compact and $co(e_1,Q)$, the convex hull of $e_1$ and $Q$, is Hausdorff. By \cite[Theorem 3.1, p.100]{inftop}, $\operatorname{Cone}(Q)$ is a Keller space and thus $\operatorname{Cone}(Q)\cong Q$.
\end{example} 

For the rest of this section, we explore some more applications of the join representation to the study of lattice and map properties.

\begin{proposition}\label{prop:order-dense-ck}
Let $X$ and $Y$ be order-dense sublattices of $C(K_1)$ and $C(K_2)$, respectively. Then
\[
\lat\bigl(\varphi_1(X)\cup \varphi_2(Y)\bigr)
\]
is order dense in $C(K_1)\ast C(K_2)$.
\end{proposition}

\begin{proof}
By Theorem~\ref{thm:CK-join}, it is enough to work in $C(K_1\ast K_2)$. Let $0<h\in C(K_1\ast K_2)$. There is a point $[x_0,y_0,t_0]$ with $0<t_0<1$ and $h[x_0,y_0,t_0]>0$. Choose neighborhoods $U_{x_0}$ of $x_0$ and $V_{y_0}$ of $y_0$, together with $\delta>0$, so that $h$ stays positive on
\[
U:=\pi(U_{x_0}\times V_{y_0}\times (t_0-\delta,t_0+\delta)).
\]
Without loss of generality, one can assume $h|_U>1$.
By order density, choose non-zero positive functions $f_1\in X$ and $g_1\in Y$ supported in $U_{x_0}$ and $V_{y_0}$, respectively.

Let us call \[
\lambda_1[x,y,t]=(1-t)f_1(x)+a\Big((1-t)f_1(x)+tg_1(y)\Big ) \]and \[\lambda_2[x,y,t]=tg_1(y)+b\Big ((1-t)f_1(x)+tg_1(y)\Big ) \] for some values $a,b\leq0$ chosen so that, after fixing $x=x_0$, $y=y_0$ and regarding the $\lambda_i$ as function of $t$, the function $0<\lambda:=(\lambda_1\wedge \lambda_2)^+\leq 1$ is such that

\[
\lambda[x_0,y_0,t]=
\begin{cases}
0 & \text{if } 0\leq t<t_0-\delta/2 \\
\lambda_2[x_0,y_0,t] & \text{if } t_0-\delta/2\leq t<t_0 \\

\lambda_1[x_0,y_0,t] & \text{if } t_0\leq t <t_0+\delta/2 \\
0 & \text{if } t_0-\delta/2\leq t\leq 1
\end{cases}
\]

Note that by continuity of $\lambda$, if we take different fixed values of $x,y$ but we keep them sufficiently close to $x_0,y_0$, the support of $\lambda[x,y,t]$ will still lie within $(t_0-\delta,t_0+\delta)$. Suppose, for example, that taking $(x,y)\in V\times W\subset U_{x_0}\times V_{y_0}\subset K\times L$ does the job, and that $\lambda[x,y,t_0]\geq\varepsilon>0$ for every $(x,y)\in V\times W$. Lastly, pick two functions $0<f_2,g_2$ in $X,Y$ that vanish outside $V,W$, respectively. It is clear then that

$$0<F:=\varphi_1(f_2)\wedge \varphi_2(g_2)\wedge\lambda\leq h,$$
because $\lambda\leq 1$ and for a certain $(x,y)\in V\times W$, $\varphi_1(f_2)\wedge \varphi_2(g_2)[x,y,t_0]>0$.
\end{proof}

The following lemma, which can be found in \cite[Lemma 3.1]{o-conv-ck}, will be needed for the next result.

\begin{lemma}\label{lem:order-ck}
For $G\subset C(K)_+$, $\inf G=0$ if and only if for every non-empty open set $U\subset K$ and every $\varepsilon>0$ there exists $x\in U$ and $g\in G$ with $g(x)<\varepsilon$.
\end{lemma} 

\begin{proposition}\label{prop:regular-ck}
For every compact Hausdorff spaces $K_1$ and $K_2$, the canonical copies of $C(K_1)$ and $C(K_2)$ are regular sublattices of $C(K_1)\ast C(K_2)$.
\end{proposition}

\begin{proof}
Identify the free product with $C(K_1\ast K_2)$. Let us write the proof for $C(K_1)$. Suppose that $f_\alpha\downarrow 0$ in $C(K_1)$. By Lemma~\ref{lem:order-ck}, for every $U\subset K_1$ and $\varepsilon>0$ we can find $x_\varepsilon\in U$ and $\alpha_\varepsilon$ such that \[
f_{\alpha_\varepsilon}(x_\varepsilon)<\varepsilon.
\] 
Note that a basis of open sets for $K_1*K_2$ consists of open sets of the form 
\[
\pi(U\times V\times I),
\]
where $U\subset K_1$ and $V\subset K_2$ are open subsets and $I\subset[0,1]$ is relatively open subinterval. So for a basic open set $\pi(U\times V\times I)$ and $\varepsilon>0$, we can choose $x_\varepsilon\in U$ and $\alpha_\varepsilon$ such that $f_{\alpha_\varepsilon}
(x_\varepsilon)<\varepsilon$. Then
\[
\varphi_1(f_{\alpha_\varepsilon})[x_\varepsilon,y,t]=(1-t)f_{\alpha_\varepsilon}(x_\varepsilon)<\varepsilon
\]
for any $y\in V$ and $t\in I$. Now, applying Lemma \ref{lem:order-ck} to the net $\varphi_1(f_\alpha)\downarrow$, we obtain the desired conclusion.
\end{proof}

\begin{proposition}\label{prop:regular-generated}
Let $X$ and $Y$ be Banach lattices. Then $X$ and $Y$ are regular sublattices of
\[
\lat\bigl(\varphi_1(X)\cup \varphi_2(Y)\bigr),
\]
inside $X\ast Y$.
\end{proposition}

\begin{proof}
Suppose $X$ is not a regular sublattice. Then there exist $x_\alpha\downarrow 0$ in $X$ and a non-zero
\[
z\in \lat\bigl(\varphi_1(X)\cup \varphi_2(Y)\bigr)
\]
such that $0<z\leq \varphi_1(x_\alpha)$ for all $\alpha$. Writing $z$ as a lattice expression in finitely many elements of $X$ and $Y$, one may pass to the principal ideals generated by the elements in $X$ together with $x_{\alpha_0}$, and the elements in $Y$. These principal ideals are lattice isomorphic to spaces $C(K_x)$ and $C(K_y)$. Inside
\[
C(K_x)\ast C(K_y)\simeq C(K_x\ast K_y),
\]
the same domination relation holds, contradicting Proposition~\ref{prop:regular-ck}.
\end{proof}

\begin{proposition}
Let $T_i:C(K_i)\to C(L_i)$ be injective lattice homomorphisms for $i=1,2$. Then the induced homomorphism
\[
T_1\overline{\ast}T_2:C(K_1\ast K_2)\to C(L_1\ast L_2)
\]
is injective.
\end{proposition}

\begin{proof}
Every injective lattice homomorphism $T_i:C(K_i)\to C(L_i)$ has the form
\[
T_if(x)=\omega_i(x)f(\sigma_i(x)),
\]
where $\omega_i\in C(L_i)_+$ has empty interior zero set and $\sigma_i(\{\omega_i>0\})$ is dense in $K_i$. The induced homomorphism on the join is therefore given by
\[
\omega([x,y,t])=(1-t)\omega_1(x)+t\omega_2(y)
\]
and
\[
\sigma([x,y,t])=
\Bigl[\sigma_1(x),\sigma_2(y),\frac{t\omega_2(y)}{(1-t)\omega_1(x)+t\omega_2(y)}\Bigr]
\]
on the set where $\omega>0$. The interior of $\{\omega=0\}$ is empty because it is contained in
\[
\{\omega_1=0\}\times \{\omega_2=0\}\times (0,1),
\]
and the image of $\sigma$ is dense because it contains $D\times (0,1)$ for a dense set $D\subset K_1\times K_2$. Hence the induced homomorphism is injective.
\end{proof}

\section{Sequece spaces and further examples}
In the previous sections we have come across some classes of spaces that are stable under free products. Being isomorphic to a free product is however quite strong, since by Proposition \ref{prop:embedding-stability} it implies having many copies of $C([0,1])$. For instance, we have the following.
\begin{proposition}
For every measure space $(\Omega,\Sigma,\mu)$ and every $1\leq p\leq \infty$, the Banach lattice $L_p(\mu)$ is not a non-trivial free product.
\end{proposition}

\begin{proof}
If $L_p(\mu)\simeq X\ast Y$ with $X,Y\neq 0$, then $X\ast Y$ is infinite-dimensional, so we may assume that $L_p(\mu)$ is infinite-dimensional.

If $1\leq p<\infty$, then Proposition~\ref{prop:embedding-stability} implies that
\[
C([0,1])\simeq \mathbb{R}\ast \mathbb{R}
\]
would embed as a closed sublattice of $L_p(\mu)$, which is impossible.

If $p=\infty$, then $L_\infty(\mu)$ is an $AM$-space with a strong unit, and any free factor must be a closed sublattice. Hence $X$ and $Y$ would be themselves $AM$-spaces. By Proposition~\ref{prop:complemented}, they are ranges of lattice projections from $L_\infty(\mu)$, so the image of the strong unit by the projections is a strong unit in each factor. Therefore there are compact Hausdorff spaces $K$ and $L$ such that
\[
X\simeq C(K),
\qquad
Y\simeq C(L).
\]
Now $L_\infty(\mu)\simeq C(K_\mu)$ for a Stonean compact $K_\mu$, whereas Theorem~\ref{thm:CK-join} gives
\[
C(K_\mu)\simeq L_\infty(\mu)\simeq C(K)\ast C(L)\simeq C(K\ast L).
\]
Since $K\ast L$ is connected while a Stonean compact carrying an infinite-dimensional $C(K)$-space is disconnected, this is impossible.
\end{proof}
Now suppose that $X$ and $Y$ are Banach lattices whose orders are given by unconditional normalized bases $(x_m)$ and $(y_n)$. For $m,n\in \mathbb{N}$, let $X_m$ and $Y_n$ be the finite-dimensional sublattices generated by the first $m$ and $n$ basis vectors, respectively.

\begin{lemma}\label{lemma:kb-finite}
For each $m,n\in \mathbb{N}$, the free product $X_m\ast Y_n$ is lattice isometrically isomorphic to the Banach lattice
\[
Z_{m,n}(X,Y)=
\Bigl\{(f_{i,j})_{i=1,j=1}^{m,n}\subset C([0,1]):
f_{i,j}(0)=f_{k,j}(0)\quad\forall k\neq i,\ f_{i,j}(1)=f_{i,l}(1)\quad\forall j\neq l \Bigr\},
\]
equipped with the coordinatewise order and with norm
\begin{equation}\label{*}\tag{$*$}
\|(f_{i,j})\|
=
\inf\Bigl\{
\Big\|\sum_{i=1}^m a_i x_i\Big\|_X+\Big\|\sum_{j=1}^n b_j y_j\Big\|_Y:
|f_{i,j}|\leq L(a_i,b_j)\ \forall i,j
\Bigr\}.
\end{equation}
\end{lemma}

\begin{proof}
Let $D_m=\{p_1,\dots,p_m\}$ and $D_n=\{q_1,\dots,q_n\}$ be finite discrete compact spaces. Since $X_m\simeq C(D_m)$ and $Y_n\simeq C(D_n)$ as Banach lattices, Theorem~\ref{thm:CK-join} yields
\[
X_m\ast Y_n\simeq C(D_m\ast D_n).
\]
The join $D_m\ast D_n$ can be represented as the complete bipartite graph $K_{m,n}$, where each edge is replaced by a copy of $[0,1]$, so every continuous function on it can be identified with an array $(f_{i,j})$ of continuous functions on $[0,1]$ satisfying the compatibility conditions $f_{i,j}(0)=f_{k,j}(0)$ $\forall k\neq i$ and $\ f_{i,j}(1)=f_{i,l}(1)$ $\forall j\neq l $.

Under this identification, the element $\sum_{i=1}^m a_i x_i\in X_m$ corresponds to the family $(L(a_i,0))_{i,j}$, and $\sum_{j=1}^n b_j y_j\in Y_n$ corresponds to $(L(0,b_j))_{i,j}$. The formula \eqref{*} defines a lattice norm such that 
\[
\|(f_{i,j=1}^{m,n})\|_\infty\leq\|(f_{i,j=1}^{m,n})\|\leq2\|(f_{i,j=1}^{m,n})\|_\infty,
\]
and it agrees with the original norms on these two canonical copies by the same argument used in Theorem~\ref{thm:CK-join}. Hence Proposition~\ref{prop:maximal-norm} shows that this is exactly the free product norm.
\end{proof}

For each element of $Z_{m,n}(X,Y)$, the infimum in \eqref{*} is attained by compactness.

Let $Z(X,Y)$ be the set of all double sequences $(f_{i,j})_{i,j\in\mathbb N }$ of functions in $C([0,1])$ such that:
\begin{enumerate}
\item for each fixed $i$, the sequence $(f_{i,j}(0))_j$ is constant;
\item for each fixed $j$, the sequence $(f_{i,j}(1))_i$ is constant;
\item there exist $\sum_{i=1}^\infty a_i x_i\in X$ and $\sum_{j=1}^\infty b_j y_j\in Y$ with $a_i,b_j\geq 0$ such that
\[
|f_{i,j}|\leq L(a_i,b_j)
\qquad (i,j\in \mathbb{N}).
\]
\end{enumerate}
Equip $Z(X,Y)$ with the lattice norm
\[
\|(f_{m,n})\|
=
\inf\Bigl\{
\Big\|\sum_{i=1}^\infty a_i x_i\Big\|_X+\Big\|\sum_{j=1}^\infty b_j y_j\Big\|_Y:
|f_{i,j}|\leq L(a_i,b_j)\ \forall i,j
\Bigr\}.
\]

The spaces $Z_{m,n}(X,Y)$ embed lattice isometrically into $Z(X,Y)$ by completing with identically zero functions. Moreover, the coordinate truncations $P_{m,n}:Z(X,Y)\to Z_{m,n}(X,Y)$ are contractive lattice homomorphisms.

Let us prove that $Z(X,Y)$ is a Banach lattice by showing that every absolutely convergent series converges. Let $(f^{(k)})=((f_{m,n})^{(k)}))_k$ be a sequence in $Z(X,Y)$ such that \[
\sum_{k=1}^\infty\|f^{(k)}\|<\infty
\]  Choose dominating series $a^{(k)}=\sum_{i=1}^\infty a_i^{(k)}x_i$ and $b^{(k)}=\sum_{j=1}^\infty b_j^{(k)}y_j$ for each $f^{(k)}$ such that
\[
\|a^{(k)}\|_X+\|b^{(k)}\|_Y\leq\|f^{(k)}\|+\frac{1}{2^k}
\]
Therefore
\[
\sum_{k=1}^\infty \|a^{(k)}\|,\quad\sum_{k=1}^\infty \|b^{(k)}\|<\infty,
\]
so
\[
\sum_{k=1}^\infty a^{(k)}=:\sum_{i=1}^\infty a_ix_i=a,\quad \sum_{k=1}^\infty b^{(k)}=:\sum_{j=1}^\infty b_jy_j=b,
\]
converge respectively in $X$ and $Y$. Now fix $i,j\in \mathbb N$. Observe that $a_i=\sum_{k=1}^\infty a_i^{(k)}$, $b_j=\sum_{k=1}^\infty b_j^{(k)}$, and \[
\|f_{i,j}^{(k)}\|_\infty\leq a_i^{(k)}+b_j^{(k)}\implies\sum_{k=1}^\infty \|f_{i,j}^{(k)}\|\leq a_i+b_j<\infty.
\]
So for every $i,j\in\mathbb N$ we have an absolutely convergent series
\[
f_{i,j}:=\sum_{k=1}^\infty f_{i,j}^{(k)}\in C([0,1]).
\]
Note that since each summand $f^{(k)}$ satisfies $(1),(2)$, so does the series $f=(f_{i,j})$. Also, since 
\[
|f_{i,j}|\leq \sum_{k=1}^\infty|f_{i,j}^{(k)}|\leq L(a_i,b_j),
\]
 $a\in X,b\in Y$ are dominant sequences for $f=(f_{i,j})$. Therefore $f\in Z(X,Y)$, and the fact that $f^{(k)}\to f$ is straightforward. We then have the following Proposition.

\begin{proposition}\label{prop:seqfreeprod}
If $X$ and $Y$ are Banach lattices which orders are given by 1-unconditional normalized bases, then $X\ast Y$ is lattice isometrically isomorphic to the Banach lattice $Z(X,Y)$ just defined.
\end{proposition}

\begin{proof}
The unions $\bigcup_{m,n}(X_m\ast Y_n)$ and $\bigcup_{m,n}Z_{m,n}(X,Y)$ are dense sublattices of $X\ast Y$ and $Z(X,Y)$, respectively. By Lemma ~\ref{lemma:kb-finite}, these dense sublattices are lattice isometrically isomorphic. Passing to completions gives the result.
\end{proof}

\begin{proposition}\label{prop:order-dense-unconditional}
If the order on $X$ and $Y$ is given by unconditional bases, then
\[
\lat\bigl(\varphi_1(X)\cup \varphi_2(Y)\bigr)
\]
is order dense in $X\ast Y$. In particular, $X$, $Y$ are regular sublattices of $X\ast Y$.
\end{proposition}

\begin{proof}
Using the representation $X\ast Y\simeq Z(X,Y)$ in Proposition \ref{prop:seqfreeprod}, let $0<f=(f_{m,n})\in Z(X,Y)$. There exist indices $i,j$ such that $f_{i,j}>0$. Choose a positive piecewise linear function $\Phi(L(1,0),L(0,1))$ with
\[
0<\Phi(L(1,0),L(0,1))\leq f_{i,j}.
\]
Then
\[
0<\Phi(\varphi_1(x_i),\varphi_2(y_j))\leq f,
\]
which proves order density. The regularity statement follows from Proposition~\ref{prop:regular-generated}.
\end{proof}
\section{$p$-Convexity}\label{sec:pconvex}

Recall that a Banach lattice $X$ is \emph{$p$-convex} with constant $C$ if
\[
\Big\|\Big(\sum_{k=1}^n |x_k|^p\Big)^{1/p}\Big\|
\leq
C\Big(\sum_{k=1}^n \|x_k\|^p\Big)^{1/p}
\]
for all $x_1,\dots,x_n\in X$ when $1\leq p<\infty$, and
\[
\Big\|\bigvee_{k=1}^n |x_k|\Big\|
\leq
C\max_{1\leq k\leq n}\|x_k\|
\]
when $p=\infty$.

If $E$ is a Banach space and $T:E\to X$ is a linear operator, we write $M^{(p)}(T)$ for the least constant $M$ such that
\[
\Big\|\Big(\sum_{k=1}^n |Tx_k|^p\Big)^{1/p}\Big\|
\leq
M\Big(\sum_{k=1}^n \|x_k\|^p\Big)^{1/p}
\]
for all finite families $(x_k)\subseteq E$ when $1\leq p<\infty$, and analogously for $p=\infty$.

\begin{proposition}\label{prop:pconvex-factorization}
Let $1\leq p\leq \infty$, let $E_1$ and $E_2$ be Banach spaces, let $X$ be a Banach lattice, and let $T_i:E_i\to X$ be $p$-convex operators. Set
\[
M=\max\{M^{(p)}(T_1),M^{(p)}(T_2)\}.
\]
Then there exist a $p$-convex Banach lattice $W$ with $p$-convexity constant $1$, contractions $R_i:E_i\to W$, and a lattice homomorphism $\gamma:W\to X$ such that
\[
T_i=\gamma R_i
\qquad (i=1,2)
\]
and
\[
\|\gamma\|\leq 2^{1-1/p}M.
\]
\end{proposition}

\begin{proof}
We treat the case $1\leq p<\infty$; the case $p=\infty$ is analogous, replacing $p$-sums by suprema.

Let $S$ be the set of all $y\in X$ satisfying
\[
|y|\leq
\Big(\sum_{j=1}^m |T_1u_j|^p+\sum_{k=1}^n |T_2v_k|^p\Big)^{1/p}
\]
for some finite families $(u_j)\subseteq E_1$ and $(v_k)\subseteq E_2$ with
\[
\sum_{j=1}^m\|u_j\|^p+\sum_{k=1}^n\|v_k\|^p\leq 1.
\]
Let $\|\cdot\|_W$ be the Minkowski functional of the closed solid set $\overline{S}$ and let
\[
W=\{y\in X:\ \|y\|_W<\infty\}.
\]
As in the proof of \cite[Theorem~3]{factorpconvex}, $W$ is a Banach lattice and is $p$-convex with constant $1$.

Define $R_i:E_i\to W$ by
\[
R_i x=T_i x,
\qquad i=1,2,
\]
regarding $T_i x$ as an element of $W$. These maps are contractions because $T_i(B_{E_i})\subseteq S$. Let $\gamma:W\to X$ be the formal inclusion. Then $\gamma$ is a lattice homomorphism and $T_i=\gamma R_i$.

It remains to estimate $\gamma$. Let $z\in W$ and choose $\lambda>\|z\|_W$. Then $z\in \lambda \overline{S}$, so by definition of $S$ there exist finite families $(u_j)\subseteq E_1$ and $(v_k)\subseteq E_2$ such that
\[
|z|\leq
\Big(\sum_{j=1}^m |T_1u_j|^p+\sum_{k=1}^n |T_2v_k|^p\Big)^{1/p}
\]
and
\[
\sum_{j=1}^m\|u_j\|^p+\sum_{k=1}^n\|v_k\|^p\leq \lambda^p.
\]
Hence
\begin{align*}
\|z\|_X
&\leq
\Big\|
\Big(\sum_{j=1}^m |T_1u_j|^p+\sum_{k=1}^n |T_2v_k|^p\Big)^{1/p}
\Big\| \\
&\leq
\Big\|\Big(\sum_{j=1}^m |T_1u_j|^p\Big)^{1/p}\Big\|
+
\Big\|\Big(\sum_{k=1}^n |T_2v_k|^p\Big)^{1/p}\Big\| \\
&\leq
M\Big(\sum_{j=1}^m \|u_j\|^p\Big)^{1/p}
+
M\Big(\sum_{k=1}^n \|v_k\|^p\Big)^{1/p} \\
&\leq
2^{1-1/p}M
\Big(
\sum_{j=1}^m \|u_j\|^p+\sum_{k=1}^n \|v_k\|^p
\Big)^{1/p} \\
&\leq 2^{1-1/p}M\lambda.
\end{align*}
Taking the infimum over $\lambda>\|z\|_W$ yields
\[
\|\gamma z\|_X\leq 2^{1-1/p}M\|z\|_W.
\]
\end{proof}

\begin{corollary}\label{corollary:factFPLp} In the setting of last proposition, we can choose \[
W=\FBL^{(p)}[E_1\oplus_1E_2] 
\]
and 
$R_i=\delta_i$, where $\delta_i$ is the embedding 
\[
E_i\overset{\iota_i}{\rightarrow}E_1\oplus_1E_2\overset{\delta}{\rightarrow} \FBL^{(p)}[E_1\oplus_1 E_2] \qquad (i=1,2).
\]

\end{corollary}
\begin{proof}
The linear operator 
\[
R:=R_1\oplus R_2:E_1\oplus_1 E_2\rightarrow W
\]
extends to a contractive lattice homomorphism
\[
\widehat{R}:\FBL^{(p)}[E_1\oplus_1 E_2]\rightarrow W.
\]
Taking \[
\widehat T:=\gamma\widehat R
\]
does the job.
\end{proof}

\begin{proposition}\label{prop:pconvex-freeproduct}
Let $1\leq p\leq \infty$. If $X_1$ and $X_2$ are $p$-convex Banach lattices with constant $1$, then $X_1\ast X_2$ is $p$-convex with constant $2^{1-1/p}$.
\end{proposition} 

\begin{proof}
Let
\[
\varphi_1:X_1\to X_1\ast X_2,
\qquad
\varphi_2:X_2\to X_1\ast X_2
\]
be the canonical inclusions. Since lattice homomorphisms commute with the lattice functional calculus, both $\varphi_1$ and $\varphi_2$ are $p$-convex with constant $1$. We also have the linear embeddings
\[
\delta_1:X_1\to \FBL^{(p)}[X_1\oplus_1 X_2],
\qquad
\delta_2:X_2\to \FBL^{(p)}[X_1\oplus_1 X_2].
\]
Taking the quotient by the closed ideal $J$ generated by the elements of the form \[
|\delta_1x_1|-\delta_1|x_1|,\qquad |\delta_2x_2|-\delta_2|x_2|
\]
of $\FBL^{(p)}[X_1\oplus_1 X_2]$ and composing the $\delta_i$ with the quotient map, we obtain contractive lattice homomorphisms \[ 
\psi_1:X_1\to Z:=\FBL^{(p)}[X_1\oplus_1 X_2]/J, \qquad\psi_2:X_2\to Z.
\]

Applying Corollary ~\ref{corollary:factFPLp} to $T_1=\varphi_1$ and $T_2=\varphi_2$, we obtain a lattice homomorphism $\widehat T$ inducing another lattice homomorphism $T:Z\to X_1\ast X_2$ such that
\[
\varphi_i=T \psi_i
\qquad (i=1,2).
\] 
and
\[
\|T\|\leq 2^{1-1/p}.
\]

By the universal property of the free product, there exists a contractive lattice homomorphism
\[
\psi:X_1\ast X_2\to Z
\]
with $\psi\varphi_i=\psi_i$ for $i=1,2$. The situation is summarized in this commutative diagram

\[
\xymatrix{& Z \\ & X_1*X_2 \ar[u]^\psi \\ X_1 \ar@/^1.5pc/@[black][ruu]^{\psi_1} \ar[ru]^{\varphi_1} \ar[r]^{\psi_1} & Z \ar[u]^{T} & X_2 \ar[lu]_{\varphi_2} \ar[l]_{\psi_2} \ar@/_1.5pc/@[black][luu]_{\psi_2}}
\]
Then
\[
 T \psi\varphi_i= T \psi_i=\varphi_i
\qquad (i=1,2),
\]
so uniqueness in the free product gives $T \psi=\mathrm{id}_{X_1\ast X_2}$.

On the other hand,
\[
\psi T \psi_i=\psi\varphi_i=\psi_i
\qquad (i=1,2).
\]
But $Z$ is generated as a Banach lattice by the ranges of $\psi_1$ and $\psi_2$, so the latter identities imply
\[
\psi T=\mathrm{id}_{Z}.
\]
Hence $T$ is a lattice isomorphism from $Z$ onto $X_1\ast X_2$, with inverse $\psi$. Since $Z$ is $p$-convex with constant $1$, we conclude that $X_1\ast X_2$ is $p$-convex with constant at most
\[
\|T\|\,\|\psi\|\leq 2^{1-1/p}.
\]
\end{proof}

\begin{remark}
Naturally, the Banach lattice $\FBL^{(p)}[X_1\oplus_1 X_2]/J$ in last proof is nothing but the free product of $X_1$ and $X_2$ in the subcategory of $p$-convex Banach lattices. Proposition \ref{prop:pconvex-freeproduct} tells us that it is isomorphic to the usual free product $X_1*X_2$, with a constant growing to infinity with the number of factors. That constant is in fact sharp: consider for instance $\bigfree_{i=1}^n\mathbb{R}\simeq C((S_{\ell_\infty^n})_+)$. The function \[
(\sum_{i=1}^n \delta_{e_i}^p)^{1/p}(t_1,\ldots,t_n)=(\sum_{i=1}^n t_i^p)^{1/p}
\]
has free product norm $n$ (see Proposition \ref{prop:Rn}), while \[
(\sum_{i=1}^n\|\delta_{e_i}\|^p)^{1/p}=n^{1/p}.
\]
This implies that the $p$-convexity constant grows to infinity with the number $n$ of free factors, and when $n=2$, it cannot be smaller than $2^{1-1/p}$ in general, so the constant in Proposition \ref{prop:pconvex-freeproduct} is sharp.
\end{remark}

In direct analogy with $p$-convexity, a Banach lattice $X$ is said to satisfy an \emph{upper $p$-estimate} with constant $C>0$ if
\[
\Big\|\bigvee_{k=1}^n |x_k|\Big\|
\leq
C\Big(\sum_{k=1}^n \|x_k\|^p\Big)^{1/p}
\]
for every finite family $x_1,\dots,x_n\in X$.

\begin{proposition}
Let $1<p<\infty$. If $X_1$ and $X_2$ satisfy upper $p$-estimates with constant $1$, then $X_1\ast X_2$ also satisfies an upper $p$-estimate.
\end{proposition}

\begin{proof}
Consider the operator
\[
T=\varphi_1\oplus \varphi_2:X_1\oplus_1 X_2\to X_1\ast X_2.
\]
Since $\varphi_1$ and $\varphi_2$ are lattice homomorphisms and the factors satisfy upper $p$-estimates, the map $T$ is $(p,\infty)$-convex in the sense used in \cite[Corollary~9.37]{freebanlat}. Therefore that result yields a factorization
\[
X_1\oplus_1 X_2 \xrightarrow{\ \delta\ } \FBL^{\uparrow p}[X_1\oplus_1 X_2] \xrightarrow{\ \widehat{T}\ } X_1\ast X_2
\]
through a free object in the category of Banach lattices with upper $p$-estimates; see also \cite{p-convex-renorming-factorization}.

Let $J$ be the closed ideal of $\FBL^{\uparrow p}[X_1\oplus_1 X_2]$ generated by
\[
|\delta(x_1,0)|-\delta(|x_1|,0)
\qquad\text{and}\qquad
|\delta(0,x_2)|-\delta(0,|x_2|),
\]
and set
\[
W=\FBL^{\uparrow p}[X_1\oplus_1 X_2]/J.
\]
Since quotients preserve upper $p$-estimates, $W$ satisfies an upper $p$-estimate as well. The quotient map induces canonical lattice homomorphisms
\[
\psi_1:X_1\to W
\qquad\text{and}\qquad
\psi_2:X_2\to W,
\]
and the argument used in Proposition~\ref{prop:pconvex-freeproduct} shows that $W$ is lattice isomorphic to $X_1\ast X_2$. Hence $X_1\ast X_2$ inherits an upper $p$-estimate.
\end{proof}

\section{Free Factors of Free Banach Lattices}\label{sec:freefactors}

In this Section, we study free factors of free Banach lattices, with particular focus on the case $\FBL[\ell_1]$.

\begin{proposition}
Let $E$ be a Banach space such that $E$ and $\ell_1(E)$ are linearly isomorphic. Then $\FBL[E]$ is linearly isomorphic to $\ell_1(\FBL[E])$.
\end{proposition}

\begin{proof}
By Proposition~\ref{prop:fbl-preserves}, we have the linear isomorphism
\[
\FBL[E]\simeq \bigfree_{n=1}^\infty \FBL[E].
\]
Proposition~\ref{prop:complemented} shows that the closed linear span of the canonical copies of $\FBL[E]$ inside the countable free product is linearly isometric to $\ell_1(\FBL[E])$ and complemented there. Hence $\ell_1(\FBL[E])$ is complemented in $\FBL[E]$. Pe\l czy\'nski's decomposition method \cite[Theorem~2.2.3]{albiackalton} now yields the linear isomorphism
\[
\FBL[E]\simeq \ell_1(\FBL[E]).
\]
\end{proof}
As a consequence, $\FBL[\ell_1]\simeq \ell_1(\FBL[\ell_1])$ as Banach spaces, and $\FBL[\FBL[\ell_1]]\simeq \bigfree_{i=1}^\infty\FBL[\FBL[\ell_1]]$ as Banach lattices.

We also have a Pe\l czy\'nski's decomposition method for free products of Banach lattices.

\begin{proposition}\label{prop:pelczynski-freeproducts}
Suppose we have Banach lattices $X$, $Y$, $X_1$, $Y_1$ and lattice isomorphisms
\[
X\simeq Y\ast X_1,
\qquad
Y\simeq X\ast Y_1,
\qquad
X\simeq \bigfree_{n=1}^\infty X.
\]
Then $X\simeq Y$.
\end{proposition}

\begin{proof}
Using commutativity and associativity of free products, we obtain
\[
Y\simeq X\ast Y_1
\simeq
\Big(\bigfree_{n=1}^\infty X\Big)\ast Y_1
\simeq
X\ast
\Big(\bigfree_{n=1}^\infty X\Big)\ast Y_1
\simeq
X\ast Y.
\]
Similarly,
\[
X\simeq \bigfree_{n=1}^\infty X
\simeq
\bigfree_{n=1}^\infty (Y\ast X_1)
\simeq
\Big(\bigfree_{n=1}^\infty Y\Big)\ast
\Big(\bigfree_{n=1}^\infty X_1\Big)
\simeq
Y\ast\Big(\bigfree_{n=1}^\infty Y\Big)\ast
\Big(\bigfree_{n=1}^\infty X_1\Big) \simeq Y\ast X.
\]
Combining the two isomorphisms gives $X\simeq Y$.
\end{proof}

As a corollary of Proposition \ref{prop:complemented},
we obtain
\[
\bigfree_{n=1}^\infty \FBL[\ell_1]
\simeq
\FBL[\ell_1].
\]
lattice isometrically. Pe\l czy\'nski's decomposition method for free products yields the following rigidity result on free factors of $\FBL[\ell_1]$ which are free Banach lattices.

\begin{proposition}\label{prop:prime-freebanach}
If we have a lattice isomorphism
\[
\FBL[\ell_1]\simeq \FBL[E]\ast \FBL[F]
\]
for Banach spaces $E$ and $F$, then
\[
\FBL[\ell_1]\simeq \FBL[E]
\qquad\text{or}\qquad
\FBL[\ell_1]\simeq \FBL[F].
\]
\end{proposition}

\begin{proof}
By Proposition~\ref{prop:fbl-preserves},
\[
\FBL[\ell_1]\simeq \FBL[E]\ast \FBL[F]\simeq \FBL[E\oplus_1 F].
\]
By \cite[Theorem~9.20]{freebanlat}, the space $E\oplus_1 F$ contains a complemented copy of $\ell_1$. Lemma~\ref{lem:comp-l1} implies that either $E$ or $F$ contains a complemented copy of $\ell_1$. By symmetry we may assume that $E$ does.

Since $\ell_1\simeq \ell_1\oplus_1\ell_1$, Pe\l czy\'nski's decomposition method \cite[Theorem~2.2.3]{albiackalton} gives a linear isomorphism
\[
E\simeq E\oplus_1 \ell_1.
\]
Applying Proposition~\ref{prop:fbl-preserves} again, we obtain
\[
\FBL[E]\simeq \FBL[E\oplus_1\ell_1]\simeq \FBL[E]\ast \FBL[\ell_1].
\]
Now Proposition~\ref{prop:pelczynski-freeproducts}, applied with
\[
X=\FBL[\ell_1],
\qquad
Y=\FBL[E],
\qquad
X_1=\FBL[F],
\qquad
Y_1=\FBL[E],
\]
yields
\[
\FBL[\ell_1]\simeq \FBL[E].
\]
\end{proof}

It would be reasonable to expect that all free factors of a free Banach lattice are themselves free Banach lattices. That is, whenever $\FBL[E]$ is isomorphic to a free product $X*Y$, then $X$ is isomorphic to $\FBL[F]$ for some Banach space $F$. However, this is not true. In fact, many non-free $C(K)$ spaces arise as free factors, as we shall see.

The following theorem (\cite[p.108]{Doublesusptheo}), often referred to as the Double Suspension Theorem, may be known to topologists and opens the door to many ``exotic'' factors of free Banach lattices.

\begin{theorem} The double suspension of any homology sphere is homeomorphic to a sphere. More specifically, if $M$ is a homology sphere of dimension $n$, then $S(SM))\cong S^1*M\cong S^{n+2}$.
\end{theorem}

Now, just take a homology sphere $M$ of dimension $n$. Then,
\begin{align*}
    \FBL[\mathbb{R}^{n+3}]\simeq C(S^{n+2})\simeq C(S^1)*C(M)
\end{align*}

Since homology spheres have dimension $\geq3$, using that finite dimensional subspaces are complemented together with Proposition \ref{prop:fbl-preserves}, we obtain

\begin{proposition}\label{prop:nonfree-factor}
If $\dim(E)\geq 6$, then $\FBL[E]$ has a free factor that is not isomorphic to a free Banach lattice.
\end{proposition}

The preceding argument can be strengthened to show that a free Banach lattice may decompose as a free product of two non-free Banach lattices. If $M$ is a homology sphere of dimension $n$, $SM$ is not a sphere (it is not even a manifold \cite[Example 10.4.9]{Davis}) but 
\begin{align*}
C(SM)*C(SM)&\simeq C((S^0*M)*(S^0*M)) \\
&\simeq C((S^1*M)*M))\simeq C(S^{n+2}*M) \\ 
&\simeq C(S^n*(S^1*M))\simeq C(S^{2n+3}).
\end{align*}

so both factors fail to be free Banach lattices while their free product is.

\section{Free products of free Banach lattices generated by lattices}\label{sec:freelattice}

Let $\mathbb L$ be a lattice. The free Banach lattice generated by $\mathbb L$, introduced in \cite{freebanlatlat}, is a Banach lattice $\FBL\langle \mathbb L\rangle$ together with a lattice homomorphism
\[
\phi_{\mathbb L}:{\mathbb L}\to \FBL\langle \mathbb L\rangle
\]
such that whenever $X$ is a Banach lattice and $T:{\mathbb L}\to X$ is a lattice homomorphism with norm bounded range, there exists a unique Banach lattice homomorphism
\[
\widehat{T}:\FBL\langle \mathbb L\rangle\to X
\]
with $\widehat{T}\phi_{\mathbb L}=T$ and
\[
\|\widehat{T}\|=\sup_{x\in\mathbb L}\|Tx\|.
\]

The free Banach lattice generated by a lattice also admits a concrete function-space representation. If
\[
L^*=\{x^*:{\mathbb L}\to [-1,1]: x^*\text{ is a lattice homomorphism}\},
\]
then for each $x\in \mathbb L$ the evaluation map
\[
\dot{\delta}_x:{\mathbb L}^*\to \mathbb{R},
\qquad
\dot{\delta}_x(x^*)=x^*(x),
\]
belongs to a Banach lattice of functions on $\mathbb L^*$. More precisely, for $f\in \mathbb{R}^{\mathbb L^*}$ define
\[
\|f\|
=
\sup\Bigl\{\sum_{i=1}^n |f(x_i^*)|:\ x_1^*,\dots,x_n^*\in \mathbb L^*,\ \sup_{x\in L}\sum_{i=1}^n |x_i^*(x)|\leq 1\Bigr\}.
\]
Then $\FBL\langle \mathbb L\rangle$ is the closed sublattice generated by the evaluation maps inside the lattice of all functions $f$ with $\|f\|<\infty$.

Some care is needed when $\mathbb L$ is not distributive. Although $\FBL\langle \mathbb L\rangle$ is defined for every lattice, the canonical map $\phi_{\mathbb L}$ is injective precisely in the distributive case. By \cite[Proposition~3.2]{freebanlatlat}, one may always replace $\mathbb L$ with its image under $\phi_{\mathbb L}$ without changing $\FBL\langle \mathbb L\rangle$.

Equivalently, one may quotient by the congruence generated by the distributive identities. Following \cite{Gratzer}, let $\alpha$ be the congruence generated by the distributive laws and write
\[
\mathbb L_D=\mathbb L/\alpha.
\]
Then $\mathbb L_D$ is distributive, every lattice homomorphism from $\mathbb L$ into a distributive lattice factors through $\mathbb L_D$, and
\[
\FBL\langle \mathbb L\rangle=\FBL\langle \mathbb L_D\rangle.
\]
Thus, when studying free products of lattices through $\FBL\langle \cdot\rangle$, there is no loss in assuming the lattices are distributive.

If $(\mathbb L_i)_{i\in I}$ is a family of lattices, their lattice free product $\bigfree \mathbb L_i$ (see \cite[Chapter VII ]{Gratzer}) need not be distributive even when the factors are. There is also the notion of distributive free product, the coproduct for the class of distributive lattices, $\bigfree_D\mathbb L_i$, which is nothing but $(\bigfree \mathbb L_i)_D$. By the preceding observation, one may work with any of them.

Suppose now that $X$ is a Banach lattice and $(\mathbb L_i)$ is a family of sublattices of $X$ that is uniformly bounded, this is, there exists $K>0$ such that $\mathbb L_i\subset KB_X$ for all $i$. It can happen that the sublattice generated by the $\mathbb L_i$ is still not bounded. Take, for example, $\mathbb L_n=\{e_n\}\subset \ell_1$. It is clear that $\mathbb L_n\subset B_{\ell_1}$ for every $n\in\mathbb{N}$. However, $\|e_1\vee\ldots\vee e_n\|=n$, so the sublattice generated by the $\mathbb L_n$ is not bounded.
 
 The situation changes if we work with finite families of sublattices only. 
 
 \begin{lemma}\label{lem:bounded-sublattice} 
 Suppose $\mathbb L_1$ and $\mathbb L_2$ are norm bounded sublattices of a Banach lattice $X$. Then the sublattice (not the vector sublattice) generated by $\mathbb L_1$ and $\mathbb L_2$ in $X$ is norm bounded.
 \end{lemma} 
 
 \begin{proof} 
 Suppose $\mathbb L_1\cup \mathbb L_2\subset KB_X$ for some $K>0$. It is immediate, via induction on the complexity of the formula, that every element $x$ of the sublattice generated by $\mathbb L_1$ and $\mathbb L_2$ in $X$ is inside a certain order interval of the form $[x_1'\wedge x_2',x_1\vee x_2]$, where $x_1,x_1'\in \mathbb L_1$ and $x_2,x_2'\in \mathbb L_2$ (and they may depend on $x$). But then $$|x|\leq |x_1\wedge x_2|+|x_1'\vee x_2'|\leq |x_1|+|x_1'|+|x_2|+|x_2'|,$$
 from which $\|x\|\leq 4K$.
 \end{proof}

\begin{proposition}\label{prop:lattice-freeproducts}
Let $\mathbb L_1$ and $\mathbb L_2$ be lattices, and let $\mathbb L_1\ast \mathbb L_2$ denote their lattice free product. Then
\[
\FBL\langle \mathbb L_1\ast \mathbb L_2\rangle
\simeq
\FBL\langle \mathbb L_1\rangle\ast \FBL\langle \mathbb L_2\rangle
\]
lattice isometrically.
\end{proposition}

\begin{proof}
Let
\[
\psi_i:{\mathbb L}_i\to \mathbb L_1\ast \mathbb L_2
\]
be the canonical lattice homomorphisms for $i=1,2$, and let
\[
\phi_i:{\mathbb L}_i\to \FBL\langle \mathbb L_i\rangle,
\qquad
\phi:{\mathbb L}_1\ast \mathbb L_2\to \FBL\langle \mathbb L_1\ast \mathbb L_2\rangle
\]
be the canonical maps. By the universal property of $\FBL\langle \mathbb L_i\rangle$, each composition $\phi\psi_i$ extends to a contractive lattice homomorphism
\[
\overline{\psi}_i:\FBL\langle \mathbb L_i\rangle\to \FBL\langle \mathbb L_1\ast \mathbb L_2\rangle.
\]

We claim that $\FBL\langle \mathbb L_1\ast \mathbb L_2\rangle$, together with $\overline{\psi}_1$ and $\overline{\psi}_2$, is a free product of $\FBL\langle \mathbb L_1\rangle$ and $\FBL\langle \mathbb L_2\rangle$. Let $X$ be a Banach lattice and let
\[
T_i:\FBL\langle \mathbb L_i\rangle\to X
\]
be bounded lattice homomorphisms. Define
\[
S_i=T_i\phi_i:{\mathbb L}_i\to X.
\]
Each $S_i$ has bounded range because $\phi_i(\mathbb L_i)$ is bounded and $T_i$ is bounded. By the universal property of the lattice free product, there exists a unique lattice homomorphism
\[
S:{\mathbb L}_1\ast \mathbb L_2\to X
\]
with $S\psi_i=S_i$ for $i=1,2$. By Lemma~\ref{lem:bounded-sublattice}, the range of $S$ is bounded. Hence the universal property of $\FBL\langle \mathbb L_1\ast \mathbb L_2\rangle$ gives a unique Banach lattice homomorphism
\[
\widehat{S}:\FBL\langle \mathbb L_1\ast \mathbb L_2\rangle\to X
\]
with $\widehat{S}\phi=S$. For each $i=1,2$,
\[
\widehat{S}\,\overline{\psi}_i\phi_i
=
\widehat{S}\phi\psi_i
=
S\psi_i
=
S_i
=
T_i\phi_i.
\]
By uniqueness in $\FBL\langle \mathbb L_i\rangle$, we conclude that
\[
\widehat{S}\,\overline{\psi}_i=T_i.
\]
The uniqueness of $\widehat{S}$ is immediate from the fact that $\phi(\mathbb L_1\ast \mathbb L_2)$ generates a dense sublattice of $\FBL\langle \mathbb L_1\ast \mathbb L_2\rangle$.
\end{proof}

\begin{corollary}\label{cor:lattice-join}
Suppose $\mathbb L_1$ and $\mathbb L_2$ are lattices with minimum and maximum elements. For $i=1,2$, let
\[
K_{\mathbb L_i}
=
\Big\{x^*\in \mathbb L_i^*:\ \max\{|x^*(m_i)|,|x^*(M_i)|\}=1\Big\},
\]
where $m_i$ and $M_i$ denote the minimum and maximum elements of $\mathbb L_i$. Then
\[
K_{\mathbb L_1\ast \mathbb L_2}\cong K_{\mathbb L_1}\ast K_{\mathbb L_2}.
\]
\end{corollary}

\begin{proof}
By \cite[Theorem~2.7]{freebanlatlatproj}, the lattices $\FBL\langle \mathbb L_i\rangle$ and $\FBL\langle \mathbb L_1\ast \mathbb L_2\rangle$ are lattice isomorphic to the spaces $C(K_{\mathbb L_i})$ and $C(K_{\mathbb L_1\ast \mathbb L_2})$, respectively. Proposition~\ref{prop:lattice-freeproducts} and Theorem~\ref{thm:CK-join} therefore give lattice isomorphisms
\[
C(K_{\mathbb L_1\ast \mathbb L_2})
\simeq
\FBL\langle \mathbb L_1\ast \mathbb L_2\rangle
\simeq
\FBL\langle \mathbb L_1\rangle\ast \FBL\langle \mathbb L_2\rangle
\simeq
C(K_{\mathbb L_1})\ast C(K_{\mathbb L_2})
\simeq
C(K_{\mathbb L_1}\ast K_{\mathbb L_2}).
\]
We conclude that
\[
K_{\mathbb L_1\ast \mathbb L_2}\cong K_{\mathbb L_1}\ast K_{\mathbb L_2}.
\]
\end{proof}

\section*{Acknowledgements}
Research partially supported by grants PID2024-162214NB-I00 and CEX2023-001347-S funded by MCIN/AEI/10.13039/501100011033. Martínez-Fernández was also supported by grant PREX2023-000080 funded by MCIN/AEI/10.13039/501100011033. 

The authors would like to thank E. Bilokopytov, E. García-Sánchez and J. Illescas for the many helpful conversations related to this work. They also thank Jaime J. Sánchez-Gabites for kindly providing useful references on the join of topological spaces.

\bibliographystyle{amsplain}
\bibliography{mainbib}

@article{factorpconvex,
 author = {Raynaud, Yves and Tradacete, Pedro},
 title = {Interpolation of {Banach} lattices and factorization of {{\(p\)}}-convex and {{\(q\)}}-concave operators},
 fjournal = {Integral Equations and Operator Theory},
 journal = {Integral Equations Oper. Theory},
 issn = {0378-620X},
 volume = {66},
 number = {1},
 pages = {79--112},
 year = {2010},
 language = {English},
 doi = {10.1007/s00020-009-1733-7},
 zbMATH = {5774238},
 Zbl = {1210.47062}
}

@book{albiackalton,
 author = {Albiac, Fernando and Kalton, Nigel J.},
 title = {Topics in {Banach} space theory},
 edition = {2nd revised and updated edition},
 fseries = {Graduate Texts in Mathematics},
 series = {Grad. Texts Math.},
 issn = {0072-5285},
 volume = {233},
 isbn = {978-3-319-31555-3; 978-3-319-31557-7},
 year = {2016},
 publisher = {Cham: Springer},
 language = {English},
 doi = {10.1007/978-3-319-31557-7},
 zbMATH = {6566917},
 Zbl = {1352.46002}
}

@book{diestel,
 author = {Diestel, Joseph},
 title = {Sequences and series in {Banach} spaces},
 fseries = {Graduate Texts in Mathematics},
 series = {Grad. Texts Math.},
 issn = {0072-5285},
 volume = {92},
 year = {1984},
 publisher = {Springer, Cham},
 language = {English},
 zbMATH = {3861790},
 Zbl = {0542.46007}
}

@book{burkinshaw,
 author = {Aliprantis, Charalambos D. and Burkinshaw, Owen},
 title = {Positive operators},
 edition = {Reprint of the 1985 original},
 isbn = {1-4020-5007-0},
 year = {2006},
 publisher = {Berlin: Springer},
 language = {English},
 zbMATH = {5061163},
 Zbl = {1098.47001}
}

@book{hatcher,
 author = {Hatcher, Allen},
 title = {Algebraic topology},
 isbn = {0-521-79540-0},
 year = {2002},
 publisher = {Cambridge: Cambridge University Press},
 language = {English},
 zbMATH = {2103273},
 Zbl = {1044.55001}
}

@article{projfree,
 author = {de Pagter, Ben and Wickstead, Anthony W.},
 title = {Free and projective {Banach} lattices},
 fjournal = {Proceedings of the Royal Society of Edinburgh. Section A. Mathematics},
 journal = {Proc. R. Soc. Edinb., Sect. A, Math.},
 issn = {0308-2105},
 volume = {145},
 number = {1},
 pages = {105--143},
 year = {2015},
 language = {English},
 doi = {10.1017/S0308210512001709},
 zbMATH = {6438434},
 Zbl = {1325.46020}
}

@article{amalg,
 author = {Avil{\'e}s, Antonio and Tradacete, Pedro},
 title = {Amalgamation and injectivity in {Banach} lattices},
 fjournal = {IMRN. International Mathematics Research Notices},
 journal = {Int. Math. Res. Not.},
 issn = {1073-7928},
 volume = {2023},
 number = {2},
 pages = {956--997},
 year = {2023},
 language = {English},
 doi = {10.1093/imrn/rnab285},
 zbMATH = {7652752},
 Zbl = {1518.46010}
}

@article{freebanlat,
 author = {Timur Oikhberg and Mitchell A. Taylor and Pedro Tradacete and Vladimir G. Troitsky},
 title = {Free {Banach} lattices},
 journal = {J. Eur. Math. Soc.},
 year = {2024},
 pages = {published online first},
}

@article{freebasico,
 author = {Avil{\'e}s, Antonio and Rodr{\'{\i}}guez, Jos{\'e} and Tradacete, Pedro},
 title = {The free {Banach} lattice generated by a {Banach} space},
 fjournal = {Journal of Functional Analysis},
 journal = {J. Funct. Anal.},
 issn = {0022-1236},
 volume = {274},
 number = {10},
 pages = {2955--2977},
 year = {2018},
 language = {English},
 doi = {10.1016/j.jfa.2018.03.001},
 zbMATH = {6855726},
 Zbl = {1400.46015}
}

@book{nieberg,
 author = {Meyer-Nieberg, Peter},
 title = {Banach lattices},
 fseries = {Universitext},
 series = {Universitext},
 issn = {0172-5939},
 isbn = {3-540-54201-9},
 year = {1991},
 publisher = {Berlin etc.: Springer-Verlag},
 language = {English},
 zbMATH = {51953},
 Zbl = {0743.46015}
}

@article{freebanlatlat,
 author = {Avil{\'e}s, Antonio and Rodr{\'{\i}}guez Abell{\'a}n, Jos{\'e} David},
 title = {The free {Banach} lattice generated by a lattice},
 fjournal = {Positivity},
 journal = {Positivity},
 issn = {1385-1292},
 volume = {23},
 number = {3},
 pages = {581--597},
 year = {2019},
 language = {English},
 doi = {10.1007/s11117-018-0626-x},
 zbMATH = {7098425},
 Zbl = {1437.46020}
}

@book{inftop,
 author = {Bessaga, Czeslaw and Pe{\l}czy{\'n}ski, Aleksander},
 title = {Selected topics in infinite-dimensional topology},
 fseries = {Monografie Matematyczne},
 series = {Monogr. Mat., Warszawa},
 volume = {58},
 year = {1975},
 publisher = {PWN - Panstwowe Wydawnictwo Naukowe, Warszawa},
 language = {English},
 zbMATH = {3476334},
 Zbl = {0304.57001}
}

@book{Gratzer,
 author = {Gr{\"a}tzer, George},
 title = {Lattice theory: {Foundation}},
 isbn = {978-3-0348-0017-4; 978-3-0348-0018-1},
 year = {2011},
 publisher = {Basel: Birkh{\"a}user},
 language = {English},
 doi = {10.1007/978-3-0348-0018-1},
 zbMATH = {5801871},
 Zbl = {1233.06001}
}

@article{freebanlatlatproj,
 author = {Avil{\'e}s, Antonio and Mart{\'{\i}}nez-Cervantes, Gonzalo and Rodr{\'{\i}}guez Abell{\'a}n, Jos{\'e} David and Rueda Zoca, Abraham},
 title = {Free {Banach} lattices generated by a lattice and projectivity},
 fjournal = {Proceedings of the American Mathematical Society},
 journal = {Proc. Am. Math. Soc.},
 issn = {0002-9939},
 volume = {150},
 number = {5},
 pages = {2071--2082},
 year = {2022},
 language = {English},
 doi = {10.1090/proc/15802},
 zbMATH = {7487897},
 Zbl = {1492.46015}
}

@article{Doublesusptheo,
 author = {Cannon, James W.},
 title = {Shrinking cell-like decompositions of manifolds. {Codimension} three},
 fjournal = {Annals of Mathematics. Second Series},
 journal = {Ann. Math. (2)},
 issn = {0003-486X},
 volume = {110},
 pages = {83--112},
 year = {1979},
 language = {English},
 doi = {10.2307/1971245},
 zbMATH = {3660468},
 Zbl = {0424.57007}
}

@book{combgrouptheo,
 author = {Lyndon, Roger C. and Schupp, Paul E.},
 title = {Combinatorial group theory.},
 edition = {Reprint of the 1977 ed.},
 fseries = {Classics in Mathematics},
 series = {Class. Math.},
 issn = {1431-0821},
 isbn = {3-540-41158-5},
 year = {2001},
 publisher = {Berlin: Springer},
 language = {English},
 zbMATH = {1554175},
 Zbl = {0997.20037}
}

@article{deHeviaTradacete2023ComplexFBL,
  author = {de Hevia, David and Tradacete, Pedro},
  title = {Free complex {Banach} lattices},
  journal = {Journal of Functional Analysis},
  year = {2023},
  volume = {284},
  number = {10},
  pages = {109888},
  doi = {10.1016/j.jfa.2023.109888}
}

@article{JardonLaustsenTaylorTradaceteTroitsky2022Convexity,
  author = {Jard\'on-S\'anchez, H\'ector and Laustsen, Niels Jakob and Taylor, Mitchell A. and Tradacete, Pedro and Troitsky, Vladimir G.},
  title = {Free {Banach} lattices under convexity conditions},
  journal = {Revista de la Real Academia de Ciencias Exactas, F\'isicas y Naturales. Serie A. Matem\'aticas},
  year = {2022},
  volume = {116},
  pages = {Article 15},
  doi = {10.1007/s13398-021-01155-8}
}

@article{Oikhberg2024Geometry,
  author = {Oikhberg, Timur},
  title = {Geometry of unit balls of free {Banach} lattices, and its applications},
  journal = {Journal of Functional Analysis},
  year = {2024},
  volume = {286},
  number = {8},
  pages = {110351},
  doi = {10.1016/j.jfa.2024.110351}
}

@article{GarciaSanchezLeungTaylorTradacete2025UpperP,
  author = {Garc\'ia-S\'anchez, Enrique and Leung, Dennis H. and Taylor, Mitchell A. and Tradacete, Pedro},
  title = {Banach lattices with upper $p$-estimates: free and injective objects},
  journal = {Mathematische Annalen},
  year = {2025},
  volume = {391},
  pages = {3363--3398},
  doi = {10.1007/s00208-024-03002-8}
}

@book{LT2,
 author = {Lindenstrauss, Joram and Tzafriri, Lior},
 title = {Classical {Banach} spaces. {II}: {Function} spaces},
 fseries = {Ergebnisse der Mathematik und ihrer Grenzgebiete},
 series = {Ergeb. Math. Grenzgeb.},
 volume = {97},
 year = {1979},
 publisher = {Springer-Verlag, Berlin},
 language = {English},
 zbMATH = {3626044},
 Zbl = {0403.46022}
}

@article{Chatzinikolaou2023OperatorASystems,
  author = {Chatzinikolaou, Alexandros},
  title = {On coproducts of operator $\mathcal{A}$-systems},
  journal = {Operators and Matrices},
  year = {2023},
  volume = {17},
  number = {2},
  pages = {435--468},
  doi = {10.7153/oam-2023-17-30}
}

@article{Kavruk2014Nuclearity,
  author = {Kavruk, Ali Samil},
  title = {Nuclearity related properties in operator systems},
  journal = {Journal of Operator Theory},
  year = {2014},
  volume = {71},
  number = {1},
  pages = {95--156},
  doi = {10.7900/jot.2011nov16.1977}
}

@article{IoanaSpaasVigdorovich2025TraceSpaces,
  author = {Ioana, Adrian and Spaas, Pieter and Vigdorovich, Itamar},
  title = {Trace spaces of full free product $C^*$-algebras},
  journal = {Compositio Mathematica},
  year = {2025},
  volume = {161},
  number = {11},
  pages = {2947--2989},
  doi = {10.1112/S0010437X25102832}
}

@article{Chirvasitu2022Dearth,
  author = {Chirvasitu, Alexandru},
  title = {On the dearth of coproducts in the category of locally compact groups},
  journal = {Theory and Applications of Categories},
  year = {2022},
  volume = {38},
  number = {20},
  pages = {791--810}
}

@article{HerfortHofmannRusso2024ProLie,
  author = {Herfort, Wolfgang and Hofmann, Karl H. and Russo, Francesco G.},
  title = {A short note on coproducts of {Abelian} pro-{Lie} groups},
  journal = {Monatshefte f\"ur Mathematik},
  year = {2024},
  volume = {204},
  pages = {887--892},
  doi = {10.1007/s00605-023-01915-1}
}

@article{Carai2026Godel,
  author = {Carai, Luca},
  title = {Free algebras and coproducts in varieties of {G}\"odel algebras},
  journal = {The Journal of Symbolic Logic},
  year = {2026},
  pages = {1--32},
  doi = {10.1017/jsl.2026.10194}
}

@article{GluesingLuerssenJany2023QMatroids,
  author = {Gluesing-Luerssen, Heide and Jany, Benjamin},
  title = {Coproducts in categories of $q$-matroids},
  journal = {European Journal of Combinatorics},
  year = {2023},
  volume = {112},
  pages = {103733},
  doi = {10.1016/j.ejc.2023.103733}
}

@book{Davis,
 author = {Davis, Michael W.},
 title = {The geometry and topology of {Coxeter} groups},
 edition = {2nd edition},
 fseries = {Springer Monographs in Mathematics},
 series = {Springer Monogr. Math.},
 issn = {1439-7382},
 isbn = {978-3-031-91302-0; 978-3-031-91303-7},
 year = {2025},
 publisher = {Cham: Springer},
 language = {English},
 doi = {10.1007/978-3-031-91303-7},
 zbMATH = {8014470},
 Zbl = {1573.20001}
}

@article{o-conv-ck,
 author = {Bilokopytov, Eugene and Troitsky, Vladimir G.},
 title = {Order and uo-convergence in spaces of continuous functions},
 fjournal = {Topology and its Applications},
 journal = {Topology Appl.},
 issn = {0166-8641},
 volume = {308},
 pages = {9},
 note = {Id/No 107999},
 year = {2022},
 language = {English},
 doi = {10.1016/j.topol.2022.107999},
 zbMATH = {7471575},
 Zbl = {1492.46016}
}

@misc{p-convex-renorming-factorization,
 author = {Enrique Garc{\'{\i}}a-S{\'a}nchez and Denny H. Leung and Mitchell A. Taylor and Pedro Tradacete},
 title = {Banach lattices with upper {$p$}-estimates: {Renorming} and factorization},
 year = {2026},
 howpublished = {Preprint, {arXiv}:2601.11056 [math.{FA}] (2026)},
 url = {https://arxiv.org/abs/2601.11056},
 arXiv = {arXiv:2601.11056}
}

\end{document}